\documentclass[sigconf]{acmart}

\newif\ifdraft
\drafttrue
\newif\ifaftersub
\aftersubfalse

 \usepackage{setspace}

\setcopyright{acmcopyright}
\copyrightyear{2023}
\acmYear{2023}
\acmDOI{XXXXXXX.XXXXXXX}

\acmConference[PODC 2023]{ACM Symposium on Principles of Distributed Computing}{June 19--23,
  2023}{Orlando, FL}
\acmPrice{15.00}
\acmISBN{978-1-4503-XXXX-X/18/06}

\usepackage{enumerate,amsthm,amsmath,
latexsym}
\usepackage{bbm}
\usepackage[mathscr]{euscript}
\usepackage{bbm}
\usepackage{natbib}
\usepackage{bbold,dsfont}
\usepackage{listings}

\usepackage{hyperref}
\hypersetup{
    colorlinks,
    linkcolor={cyan!50!black},
    citecolor={green!50!black},
    urlcolor={blue!80!black}
}
\usepackage[nameinlink,capitalise]{cleveref}
\usepackage{framed}
\definecolor{shadecolor}{HTML}{F5F5F5}

\usepackage[utf8]{inputenc}
\usepackage{amsfonts}
\usepackage{tikz}
\usepackage{tkz-euclide}
\usepackage{xspace}
\usetikzlibrary{intersections}
\usetikzlibrary{positioning}
\usetikzlibrary{arrows}
\usetikzlibrary{shapes}
\usetikzlibrary{calc}

\usepackage{todonotes}

\newtheorem{theorem}{Theorem}[section]  

\newtheorem{lemma}[theorem]{Lemma}

\newtheorem{corollary}[theorem]{Corollary}

\newtheorem{proposition}[theorem]{Proposition}

\newtheorem{definition}[theorem]{Definition}

\newcommand{\cfirst}{purple\xspace}

\newcommand{\cthird}{cyan\xspace}

\newcommand{\crwprocess}{Random Walk Process\xspace} 

\newcommand{\dualprocess}{Diffusion Process\xspace}

\newcommand{\indi}[1]{{e^{(#1)}}}

\usepackage{xcolor}
\usepackage{mathtools}
\definecolor{cricolor}{HTML}{ED4863}
\definecolor{first Term}{HTML}{0590ff}
\definecolor{second Term}{HTML}{ff7f50}
\definecolor{third Term}{HTML}{40e0d0}
\definecolor{cfirst}{HTML}{a000a0} 
\definecolor{csecond}{HTML}{FF6D28} 
\definecolor{cthird}{HTML}{00bfc9} 
\newcommand{\Cnote}[1]{\ifdraft{\color{cricolor} Cris: #1}\fi}

\newcommand{\fnote}[1]{
\ifdraft
{\color{teal} F: #1}\fi}

\newcommand{\extend}[1]{{\ifdraft\ifaftersub\color{blue} extend: #1}\fi\fi}

\newcommand{\omitthis}[1]{}

\newcommand{\loc}[1]{\mathbbmss{L}_{\{#1\}}}
\newcommand{\particle}[1]{\mathbbmss{P}_{\{#1\}}}

\usepackage{cancel}
\newcommand{\asterisk}{\stackrel{\mathclap{\normalfont\mbox{\scalebox{.7}{(*)}}}}{=}}

\usepackage{stackengine}
\stackMath

\newcommand{\mone}{Node Model\xspace}
\newcommand{\mtwo}{Edge Model\xspace}

\newcommand{\avgprocess}{Averaging Process\xspace}

\newcommand{\E}{\mathds{E}}
\def \V{\mathds{V}ar}
\newcommand{\var}{\mathds{V}ar}
\newcommand{\Var}{\var}
\newcommand{\Prob}{\mathds{P}}
\newcommand{\R}{\mathbb R}

\def\a{\alpha}

\newcommand\twonorm[1]{\lVert#1\rVert_2^2}

\newcommand{\ind}[1]{\mbox{{\large 1}} _{\{#1\}}}
\newcommand{\one}{\textbf{1}}

\newcommand{\ol}[1]{\overline{#1}}
\newcommand{\ul}[1]{\boldsymbol{#1}}
\def\sm{\! \setminus \!}

\newcommand{\brac}[1]{\left( #1 \right)}

\newcommand{\pbrac}[1]{\langle #1 \rangle}

\def\a{\alpha}

\def\d{d_{\text{min}}}
\def\D{d_{\text{max}}}

\def\l{\lambda}

\def\R{\mathbb R}

\def\x{\xi}

\renewcommand{\epsilon}{\varepsilon}
\def\newChi{\raisebox{2pt}{$\chi$}}

\newcommand{\avg}{\operatorname{Avg}}

\author{P. Berenbrink}
\affiliation{%
  \institution{Universit\"at Hamburg}
  \streetaddress{Mittelweg 177}
  \city{Hamburg}
  \state{Germany}
  \postcode{20148}
}

\author{C. Cooper}
\affiliation{%
  \institution{King's College London}
  \streetaddress{Strand}
  \city{London}
  \state{United Kingdom}
  \postcode{WC2R 2LS}
}

\author{C. Gava}
\affiliation{%
  \institution{King's College London}
  \streetaddress{Strand}
  \city{London}
  \state{United Kingdom}
  \postcode{WC2R 2LS}
}

\author{F. Mallmann-Trenn}
\affiliation{%
  \institution{King's College London}
  \streetaddress{Strand}
  \city{London}
  \state{United Kingdom}
  \postcode{WC2R 2LS}
}

\author{T. Radzik}
\affiliation{%
  \institution{King's College London}
  \streetaddress{Strand}
  \city{London}
  \state{United Kingdom}
  \postcode{WC2R 2LS}
}

\author{D. Kohan Marzag\~{a}o}
\affiliation{%
\institution{University of Oxford}
  \streetaddress{Wellington Square}
  \city{Oxford}
  \state{United Kingdom}
  \postcode{OX1 2JD}
}

\author{N. Rivera}
\affiliation{%
\institution{University of Valpara\'iso}
  \streetaddress{Blanco 951}
  \city{Valpara\'iso}
  \state{Chile}
  \postcode{951}
}

\ccsdesc[500]{Theory of computation~Random walks and Markov chains}
\ccsdesc[500]{Theory of computation~Random network models}
\ccsdesc[500]{Mathematics of computing~Markov processes}
\ccsdesc[500]{Theory of computation~Distributed algorithms}

\title{Distributed Averaging in Opinion Dynamics}

\begin{document}

\begin{abstract}
We consider two simple asynchronous opinion dynamics on arbitrary graphs 
where every node $u$ of the graph has an initial value $\xi_u(0)$. 
In the first process, which we call the $\mone$, at each time step $t\ge 0$, 
a random node $u$ and a random sample of $k$ of its neighbours $v_1,v_2,\cdots,v_k$ are selected. 
Then, $u$ updates  its 
current value $\xi_u(t)$ to
$\xi_u(t+1) = \alpha \xi_u(t) + \frac{(1-\alpha)}{k} \sum_{i=1}^k \xi_{v_i}(t)$, where 
$\alpha \in (0,1)$ and $k\ge 1$ are parameters of the process.
In the second process, called the $\mtwo$, 
at each step a random pair of adjacent nodes $(u,v)$ is selected, and then node $u$ updates its value equivalently to the $\mone$ with $k=1$ and $v$ as the selected neighbour. 

For both processes, the values of all nodes converge to the same value $F$,
which is a random variable
depending on the random choices made in each step. 
For the $\mone$ and regular graphs,
and for the $\mtwo$ and arbitrary graphs, the expectation  of $F$ is the average of the initial values 
$\frac{1}{n}\sum_{u\in V} \xi_u(0)$. 
For the $\mone$ and non-regular graphs, the expectation of $F$ is the degree-weighted average of the initial values. 

Our results are two-fold. We consider  the concentration of $F$ and show tight bounds on the variance of $F$ for regular graphs.
We show that when the initial values do not depend on the number of nodes, 
then the variance is negligible, and hence the nodes are able to estimate the initial average of the node values.
Interestingly, this variance does not depend on the graph structure. For the proof we introduce 
a duality between our processes and a 
process of two correlated random walks.
We also analyse the convergence time for both models 
and for arbitrary graphs, showing bounds on the time $T_\varepsilon$
required to make all node values `$\varepsilon$-close' to each other. Our bounds are asymptotically tight under some assumptions
on the distribution of the initial values.
\end{abstract}

\keywords{Distributed Averaging, Dual Processes, Stochastic Processes, Random Walks}

\maketitle

\section{Introduction}
In this paper, we consider two natural  asynchronous opinion dynamics, extending the notion of pull voting 
to averaging of numeric opinions. We refer to these models as the
$\mone$ and the $\mtwo$. 
We are given a connected undirected graph $G=(V,E)$
with $n$ nodes, which we refer to as agents, and $m$ edges. Each agent $u$ has a real number as its initial value at time $0$.
In the $\mone$, at each time step $t\ge 1$,
a node $u$ is chosen uniformly at random. This node observes $k\ge 1$ of its neighbours $v_1,\ldots v_k$,
selected uniformly at random, and updates its value unilaterally 
(i.e., only $u$ updates its value at a this step) to an $\alpha$-fraction of its current value plus a $(1-\alpha)$-fraction of the average of the  values 
of $v_1,\ldots v_k$. 
In the $\mtwo$, a (directed) edge $(u,v)$ is observed uniformly at random and $u$  updates its value unilaterally to an $\alpha$-fraction of its own value plus a $(1-\alpha)$-fraction of the value of $v$. For regular graphs and $k=1$ the $\mtwo$ 
is the same as the $\mone$.
Similarly to the voter model, 
where a randomly selected node takes on the opinion of a random neighbour, 
the dynamics which we consider are very natural, possibly the simplest randomised pull-based protocols 
for reaching consensus on an opinion which is close to the average of the initial opinions. 
Agents in this protocols do not perform any complicated processing of the information obtained, other than the computation of a simple average of the received values.

In contrast to neighbourhood load balancing processes, in our model, only one selected node $u$ changes its value at each time step. 
This is another aspect of the simplicity of our processes, as they do not require coordination of simultaneous updates 
in two or more nodes.

It is easy to see that, over  time, the nodes' values converge to a single value, which 
we denote by $F$ and refer to as the \emph{convergence value}. 
This 
follows from observing that 
$\max_{u,v\in V}\left| \xi_u(t) - \xi_v(t)\right|$, where $\xi_x(t)$ is the value at node $x$ at the end of step $t$,
is non-increasing and tends to $0$. 
 $F$ is a random variable (depending on the random choices of the protocols) and its  expectation  
 turns out to be equal to the average of the initial values 
$\frac{1}{n}\sum_{u\in V} \xi_u(0)$
in the $\mtwo$, and the degree-weighted average of the initial values
$\frac{1}{2m}\sum_{u\in V} d_u \xi_u(0)$
in the $\mone$.
Hence we refer to either of 
these models as an \avgprocess. 

A natural question is to quantify the variance of $F$, $\var(F)$, to understand when $F$ 
is concentrated around its expected value. 
Estimating the variance and analysing the convergence rate of our protocols are the main aims of this paper.
By quantifying the variance of $F$, we analyse the ability of our processes to estimate the initial average of the node values.
In comparison to load balancing processes,
which keep the average node load as invariant, divergence of $F$ from its 
expectation can be viewed as the price of simplicity of our averaging processes.

To bound  $\var(F)$ we introduce a time-reversed dual process. This process simplifies the analysis and enables us to  derive the exact value of the variance for regular graphs, which,
under some mild assumptions on the initial distribution of the node values, implies concentration of $F$. Regarding the convergence time, we derive tight bounds for both protocols by using an appropriate potential function.  These results are summarised in~\cref{thm:mone}, with a more detailed statement regarding $\var(F)$ given 
in~\cref{thm:limitingVarExact1}.
In the  $\mone$ our bounds on both the convergence time and 
the $\var(F)$ show, somewhat surprisingly, only a negligible dependence on $k$.  It makes almost no difference if 
$k=1$ or if it is close to the node degree.
Another interesting property is that  $\var(F)$ is independent (up to some constants) of both the graph structure and 
the node mapping  of the initial values. 
For example, for the same set of initial node values and $\mone$ with $k=1$, the variance $\var(F)$ is the same for
complete graphs as for cycle graphs.
%

Processes with unilateral updates can provide a natural model in various scenarios. The graph can represent a social network where individuals (the nodes) change their opinions after observing opinions of some of their friends. For example, they might want to decide which phone to buy or how much they should budget for a given type of vacation, and check their 
friends' opinion on this. Unilateral updates model situations where an individual is influenced by the opinion of specialist friends. This does not necessarily mean that a ``specialist'' is also influenced by the opinions of all its friends (e.g. see \cite{berenbrink2017ignore}). There are numerous examples for models with unilateral updates in different areas, including: population dynamics \cite{lieberman2005evolutionary},
biology \cite{chazelle2012natural} and  
robotics \cite{beni1996pattern}. 

Note that the simple pull-based communication with unilateral updates considered in this paper cannot guarantee 
convergence to the initial (simple or weighted) average.
To guarantee such convergence, one would need a stronger communication model. 
For example, it is well-known that if two nodes average their values in a given time step 
(e.g., two neighbours in the graph in a load-balancing process), then convergence to the average is guaranteed \cite{Aldous2012Averaging}.
Such a protocol, however, requires  coordinated updates at these two nodes.
Compare this with our models where we do not have coordinated updates. 
In this context, analysing the variance of $F$ can be viewed as studying the 
cost of simplicity.
%
\paragraph{Our Results}
In this paper we show tight bounds on the convergence time of the $\mone$ and the $\mtwo$. 
Furthermore, we calculate the variance of the final value $F$ at the nodes. 
The update at step $t\ge 1$ in both models can be expressed as $\xi(t) = B(t)\xi(t-1) $, where 
$B(t) \in \mathbb{R}^{|V| \times |V|}$ is the 
\emph{communication matrix} for step $t$ and $\xi(t) \in \mathbb{R}^{|V|}$ is the vector of node values at the 
end of this step.
There are many convergence results for models which have \emph{doubly-stochastic} (average-preserving) update matrices 
(e.g., \cite{C89}). In contrast, the update matrices in our unilateral-update models do not have this property (matrices $B(t)$ are stochastic, but not doubly-stochastic) posing challenges in analysis. There are far less results for this setting \cite{DBLP:journals/automatica/GhaderiS14}.
Note, for example, that in our models the convergence value is not known ahead of time, requiring a novel potential function,  especially for non-regular graphs.\footnote{The standard potential functions which do not take the degrees into account are not decreasing in expectation.}

The main technical novelty of our paper lies, however, in the development of  
methods for obtaining asymptotically tight concentration bounds (as detailed in~\cref{thm:limitingVarExact1}).
A well-known duality between 
the voter model
and coalescing random walks,\footnote{%
In coalescing random walks, one random walk starts on each node and whenever two or more random walks meet, they merge into one random walk. The coalescence time is the expected time until only one random walk is left. 
}
obtained by coupling these two processes (one being run backwards in time), shows 
that the voting time and the coalescence time have the same distribution.
We derive a generalisation of this duality
which relates the $\mone$ to a diffusion process.
In this dual diffusion process, we have $n$ different `commodities', indexed by the nodes, and start by placing the unit amount of commodity $i$ on node $i$, for each $i\in V$.
These commodities are then being diffused throughout the graph in the following steps.
In one step, first a node $u$ and its $k$ neighbours are randomly selected.
Then for each commodity which currently has positive `load' at node $u$, a $(1-\alpha)$ fraction of this load is taken from $u$ and 
 distributed in equal proportion among the selected $k$ neighbours. 
 
If we couple $T$ steps of the $\mone$ process with the above diffusion process, running one of them backwards in time, then at the end 
of these $T$ steps,
the distribution of node values in $\mone$ is the same as  
the (weighted) distribution of commodities in the diffusion process.
Further, we show that the variance of the diffusion process is the same as 
the variance of a process of two correlated random walks associated with the diffusion process.
The latter random walk process can be analysed by studying the stationary distribution of an associated  Markov chain. 
Crucially, despite the fact that 
this Markov chain is not bi-directional (and so, it does not satisfy the detailed balanced equations),
we are able to compute its exact stationary distribution. 
Using this distribution and arguing via the diffusion process, we obtain the variance of the convergence value of the $\mone$.
The formal definitions, the formal statement of the duality and the details of the analysis
are given in~\cref{sec:concentration}.

Some of the proofs are omitted, but they are included in 
the Appendix.
%
%
\section{The Model and Detailed Contribution}
%
We have a connected undirected graph $G = (V,E)$ with $n$ nodes, $m$ edges and the set of nodes
$V = \{1,2,\ldots, n\}$.
If the graph is regular we use $d$ for the node degree; otherwise $d_i$ is the degree of node $i$ and 
$\D$ and $\d$ are the  maximum and minimum node degree.
We assume that every node has an initial {\em value}, 
which is a real number.
Furthermore, $\xi_{i}(t)$ is defined as a random variable equal to the value of node $i\in V$ at the end of the $t$-th 
time step, and $\xi(t)$ as the vector of these values.  
The vector $\xi(0)=\xi$ represents the  nodes' initial values. 
%
In this paper we show results for two different models called $\mone$ and $\mtwo$. In the following we will first introduce the $\mone$ and state the corresponding results. Then, we will do the same for the $\mtwo$. 

\begin{definition}[$\mone$]
The model has two parameters: a real number 
$\alpha \in [0,1)$ and an integer $k \ge 1$. 
At each time step $t\geq 1$, 
a node $u$ is chosen uniformly at random. That node, in turn, chooses $k$ of its neighbours $v_1, v_2, \dots v_k$ uniformly 
at random and without replacement. Then, 
\[ \xi_u(t) = \alpha \cdot \xi_u(t-1) + 
\frac{(1-\alpha)}{k} \cdot \sum_{i=1}^k  \xi_{v_i}(t-1), \]
and for all $v\neq u$, $\xi_{u}(t+1) = \xi_{u}(t)$.
\end{definition}
Note that for $k=1$ and $\alpha=0$ this model is equivalent to the voter model, where a node chosen uniformly at random
takes the opinion of its random neighbour. 
It is easy to see that in the $\mone$ the values of nodes converge to the same value $F$. 
In \cref{lem:NodeMartingale}, we show that  
$\E(F)=\sum_{i=1}^{n}\frac{d_i}{2m}\cdot \xi_i(0)$.
In the following we assume w.l.o.g.\ that the initial values are {\em centered at $0$}, so that 
$\sum_{i=1}^{n}\frac{d_i}{2m}\cdot \xi_i(0)=0$. 
%
For regular graphs, this becomes $\E(F)=\frac{1}{n}\sum_i\xi_i(0))=0$. 
Let $1-\lambda_2(P)$ be the eigenvalue gap of the associated transition matrix $P$ (defined in  \cref{sec:notation}),
and for $x\in R^n$, let $\| x \|_2$ be
the Euclidean norm of $x$.

\begin{theorem}[$\mone$]\label{thm:mone}
\begin{shaded}
Let $G$ be a connected graph, 
$\alpha$ a constant in $(0,1)$ and $k\ge 1$ integer (parameters of the model).
Let $\epsilon\in(0,1)$, 
and let $T=T_\epsilon$ be the first time such that $\xi(T)$ is $\epsilon$-converged 
(defined in~\cref{sec:notation}).
\begin{enumerate}
\item(Convergence time)  Then, we have w.h.p. \[T = O\left(\frac{n \log\left(n \twonorm{\xi(0)}/\epsilon\right)}{1-\lambda_2(P)}\right).\]
Furthermore, there exists an initial state $\xi(0)$ for which this bound is, up to constants, tight for every $k\geq 1$.
\item(Concentration) If $G$ is regular 
and regardless of $k$, we have
\[ \var(F) = \Theta\left(\frac{\twonorm{\xi(0)}}{n^2}\right).\]
\end{enumerate}
\end{shaded}
\end{theorem}

The proof of the first part of the theorem can be found in \cref{app:thmMone},
where we also present some more detailed bounds on the convergence time.
Those detailed bounds indicate that the convergence time only slightly improves with increasing $k$ 
since it scales as $(1+1/k) \in [1,2]$. 

Our main result, the concentration bound, is stated in the second part of the theorem, with
more detailed bounds derived in \cref{thm:limitingVarExact1}.
%
These bounds do not depend on either the value of $k$ or any structural properties of the graphs.
Thus, for example, the variance of $F$ on the clique and the cycle are asymptotically the same.
The initial distribution of the values $\xi(0)$ also does not have impact, asymptotically. 
Note that if all initial values are $o(\sqrt{n})$, then $\var(F)=o(1)$, 
and by Chebychev's Inequality, with probability $1-o(1)$, $F = o(1)$.

Our bounds on the convergence time are similar to the bounds on the convergence time of 
neighbourhood load balancing~\cite{DBLP:journals/jpdc/BerenbrinkFH09}, which, in turn, 
can be related to the mixing time of random walks.  
Load-balancing (diffusion) bounds are often stated in terms of the \emph{discrepancy} of  $\xi(t)$
defined as $K=\max_i \xi_i(t) - \min_i \xi_i(t)$.  
Note that $\log(\twonorm{\xi(0)}) \leq 2\log(Kn)$, due to our assumption that $\sum_i\xi_i(0)=0$.
The additional factor of $n$ in our convergence bound in comparison to~\cite{DBLP:journals/jpdc/BerenbrinkFH09}
is due to the fact that we consider an asynchronous model where one node is activated at the time, while 
\cite{DBLP:journals/jpdc/BerenbrinkFH09} considers a synchronous model.

Note that our convergence time bound does not give results for the voter model since we assume that $\alpha$
is a positive constant. For comparison, an upper bound of $O\left({n}/\left(1-\lambda_2(P)\right)\right)$ on the expected convergence time of the \emph{parallel} voter model for regular graphs is given in~\cite{DBLP:conf/podc/CooperEOR12}. 
Thus, our process is faster by $\Omega(n/\log n)$, provided that 
$K$ and $1/\epsilon$ are polynomial in $n$.

\begin{definition}[$\mtwo$] 
At each time step $t\geq 1$, 
a {\em directed} edge $e=(u,v)$ is chosen uniformly among all edges. Then
$ \xi_u(t) = \alpha \xi_u(t-1) + (1-\alpha)  \xi_{v}(t-1). $
For all $u'\neq u$, $\xi_{u'}(t+1) = \xi_{u'}(t)$.
\end{definition}

We obtain the following results for the $\mtwo$. The proof of the results can be found in the Appendix.
Let $\lambda_2(L)$ be the second-smallest eigenvalue of the Laplacian of $G$ (defined in \cref{sec:notation}).

\begin{theorem}[$\mtwo$]\label{thm:mtwo}
\begin{shaded}
Let $G$ be a connected graph and $\alpha$ a constant in $(0,1)$.
Let $\epsilon\in(0,1)$,
and let $T=T_\epsilon$ be the first time such that $\xi(T)$ is $\epsilon$-converged. 
\begin{enumerate}
\item(Convergence time)  Then, we have w.h.p. 
\[T = O\left(\frac{m \log\left(n{\twonorm{\xi(0)}}/{ \epsilon}\right)}{\lambda_2(L)}\right).\]
Furthermore, there exists an initial state $\xi(0)$ for which this bound is tight up to constants.
\item(Concentration) If $G$ is regular, 
we have
\[ \var(F) = \Theta\left(\frac{\twonorm{\xi(0)}}{n^2}\right)\]
\end{enumerate}
\end{shaded}
\end{theorem}
It is interesting to note that the fixed point in the $\mtwo$ is, in expectation, the initial average---even for irregular graphs. 
The concentration follows from \cref{thm:mone} since for regular graphs the $\mone$ and the $\mtwo$ are identical (for $k=1$).
For $d$-regular graphs  both theorems (\cref{thm:mone} and \cref{thm:mtwo}) give the same bound on the convergence time. Note that in this case there is a factor of $d$ between $1-\lambda_2(P)$ and $\lambda_2(L)$. 

\section{Previous Related Results}
Our protocols fit very well into the framework of opinion dynamics on graphs. The study of these dynamics has a long history in many academic disciplines, including physics, computer science, electrical engineering, population genetics and epidemiology. 
Processes falling into this category include the voter process \cite{DBLP:conf/podc/CooperEOR12, DBLP:conf/icalp/BerenbrinkGKM16, DBLP:journals/jda/CooperR18}, gossip algorithms \cite{DBLP:conf/podc/Mosk-AoyamaS06}, opinion dynamics \cite{HK02, BhattacharyyaBCN13, DeGroot, FJ90},
consensus  \cite{DBLP:conf/icalp/BerenbrinkFGK16, DBLP:conf/podc/GhaffariP16a, DBLP:conf/wdag/CooperRRS17, DBLP:journals/sigact/BecchettiCN20}, majority protocols \cite{DBLPjournals/dc/AngluinAE08, CER14, DBLP:conf/wdag/CooperRRS17, DBLP:conf/focs/DotyEGSUS21} diffusion load balancing \cite{DBLP:phd/hal/Lambein20, alistarh_et_al:LIPIcs:2020:12414, rabani1998local}, or so-called averaging 
processes~\cite{aldous}.  

Many of these models assume unilateral updates. For example, in the voter process, in each step a random agent adopts the opinion of a randomly chosen neighbour. Many majority protocols \cite{CER14, DBLP:conf/wdag/CooperRRS17} assume that an agent randomly chooses $k$ of its neighbours and adopts the majority opinion. Our model is continuous version of this where we are allow to average between multiple opinions rather than limit the opinions to be a non-numeric set.   
In many opinion dynamics,
e.g., \cite{DeGroot, FJ90}, the individuals communicate with all their neighbours before they adjust their opinions. These models are usually deterministic. The authors of  \cite{DBLP:conf/wine/FotakisKKS18} consider a randomised variant of the Friedkin-Johnsen model which is similar to our $\mone$. The agents are only allowed to interact with a small randomly chosen subset of their neighbours. The authors refer to their model as using only limited information. The motivation of the model is as follows. 
In today's social networks many users have a very large amount of followers/friends resulting in large degree nodes. In such a setting, it seems unnatural to assume that every agent/user communicates first with all its friends before making up their own mind about a subject (e.g., the esteem they hold for a brand). Instead, they might only communicate with a small subset of their friends before taking a decision. 

\paragraph{Voting and Majority.} Our model can be regarded as a generalisation of the
voter model which was first analysed in~\cite{HP01}.
In \cite{HP01} the authors assume a network, and any node of the network has a discrete value. The authors show a bound of $O(t_m \cdot \log n)=O(n^3\log n)$ where $t_m$ is the expected  meeting time  of two random walks.
In \cite{DBLP:conf/podc/CooperEOR12} the authors provide an improved upper bound of $O((1-\lambda_2)^{-1} \log^4n+ \rho))$ on the expected consensus time for  any graph $G$, where $\lambda_2$ is the second eigenvalue of the transition matrix of a random walk on $G$, with $\rho = \big(\sum_{u\in V(G)}d(u)\big)^2/ \sum_{u\in V(G)}d^2(u)$.  The value of $\rho$ ranges from $\Theta(1)$, for the star graph, to  $n$, for regular graphs. The authors of \cite{DBLP:conf/icalp/BerenbrinkGKM16} consider voting in dynamic graphs and analyse the consensus time in terms of the conductance $\phi$  of the underlying graphs. 
The authors show a bound on the expected consensus time of 
$O(d_{max} \cdot n/(d_{min} \cdot \phi))$ for graphs with minimum degree $d_{min}$ and maximum degree $d_{max}$. A generalisation of this model for discrete opinions was introduced in \cite{CER14}. Similar to our model,  the nodes consider the opinion of $k$ randomly chosen neighbours. In \cite{CER14} the nodes adopt the majority opinion (the own opinion can be included or not), and not the average as in our models.  There are several different models with different tie breaking-rules. See \cite{DBLP:journals/sigact/BecchettiCN20} for an overview article about related consensus dynamics.
\paragraph{Opinion Dynamics.} There has been an interesting line of research trying to explain the spread of innovations and new technologies based on local interactions, where nodes are only allowed to communicate with their direct neighbours \cite{DBLP:conf/focs/MontanariS09, Draief,  DBLP:journals/siamcomp/Chazelle15, chazelle2012natural}.
The study of opinion-forming processes via local communication was introduced by DeGroot
\cite{DeGroot} where it is assumed that the underlying network is a directed graph with weighted edges. The  nodes of the network are agents having opinions modelled as real numbers. The state of the network in step $t+1$ is the product of the opinions in step $t$ times the weighted adjacency matrix. If and how fast the process converges  depends heavily on the matrix. For example, if the edge weights of node with degree $d$ are chosen as $1/d$, the convergence time is  roughly $\log n\cdot (1-\lambda_{\max})^{-1}$, where $\lambda_{\max}$ is the largest eigenvalue of the matrix (see \cite{DBLP:journals/automatica/GhaderiS14}). In \cite{particles} the author uses a similar averaging approach to model the movement of particles.  
Friedkin and Johnsen (FJ) \cite{FJ90} extended the model by incorporating private opinions. n there, every agent has a private opinion, which does not change, and a so-called expressed opinion. The expressed opinion of an agent is defined as a deterministic function of the expressed opinions of all its neighbours and its private opinion. 
Another very influential model is the one by  Hegselmann and Krause~\cite{HK02}. In this model the set of neighbours that influence a given agent is no longer fixed and the agents' opinions and their respective sets of influencing neighbours co-evolve over time. At any point in time the set of influencing neighbours of an agent consists of all the neighbours in a given static social network with an opinion close to their own opinion.  In \cite{FotakisPS16} the authors  study convergence properties of a general model where the agents update their opinions in rounds to a weighted average of all opinions in their neighbourhoods. In \cite{DBLP:conf/wine/FotakisKKS18} the authors consider a variant of the Friedkin Johnson model where the agents interact with a small subset of their neighbours. They refer to their model as using only limited information.  The authors show convergence properties of simple and natural variants of the FJ in this setting.

\paragraph{Diffusion and Consensus.}
Another research area related to our work is that of  diffusion protocols. Here a network of $n$ identical nodes is given and every node stores a value. The value can, for example, model the load of the nodes or it can simply be a number.   The protocol runs in parallel steps and in each step all nodes average their own value with the value of all neighbours. In the load balancing setting on a regular network with degree $d$, this is the same as sending a load of $\max\{0, \ell_u-\ell_v\}$ over the edge from $u$ to $v$ where $\ell_u$ ($\ell_v$) is the load of node $u$ ($v$). 
The objective is to distribute the load as evenly as possible among the nodes whilst minimizing the number of load balancing steps. If the values model numbers, diffusion can also be used to reach average consensus \cite{DBLP:phd/hal/Lambein20}.
A variant of a diffusion process is the so-called dimension  exchange process where the edges used for the balancing form a matching. 
The diffusion model was first studied by \cite{C89} and, independently, by \cite{DBLP:phd/hal/Lambein20}. The authors of \cite{MGS98} show a tight connection between the
convergence rate of the diffusion algorithm and the absolute value of the second largest eigenvalue~$\lambda_{\max}$ of the diffusion matrix $P$. P is defined as follows: $p_{ij} =1/(d+1)$ if  $\{i,j\} \in E$ where $d$ is the degree of the nodes. The convergence time is bounded by $2 \log (n^2+K)/(1-\lambda_{\max})$, where $K$ is the initial maximum load difference. There is a similar connection between $\lambda_{\max}$ and the convergence time for the dimension exchange model \cite{DBLP:journals/jpdc/BerenbrinkFH09}.

There is a vast amount of literature about diffusion processes in different research communities. In \cite{SS94} the authors observe relations between
convergence time and properties of the underlying network, like electrical and fluid conductance. In \cite{DBLP:phd/hal/Lambein20} there is a nice overview about the results. The author distinguishes between consensus (all nodes have to agree on a value) or average consensus where nodes have to calculate the average, similar to neighbourhood load balancing. The main goal of the author is to develop diffusion-type  algorithms where the nodes are only aware of their own edges, but neither of the edges of their neighbours nor of their degrees. 
The author presents algorithms for both types of consensus with convergence time $\widetilde{O}(n^4)$ and dynamic networks, together with randomised algorithms for undirected graphs. 

Many publications focus on average consensus via diffusion type processes. 
In \cite{DBLP:conf/focs/KempeDG03} the authors present a 
gossip-based algorithm (which is similar to an averaging algorithm) for complete graphs.  The authors show that the protocol can be used  for the computation of sums, averages, random samples, quantiles, and other aggregate functions. They also show that our protocols converge exponentially fast. In \cite{DBLP:journals/tit/BoydGPS06, DBLP:conf/isit/BenezitBTTV10} the authors generalise the result to arbitrary graphs. In \cite{DBLP:journals/tit/BoydGPS06} the authors show that the averaging time of their algorithm depends on the second largest eigenvalue of a doubly stochastic matrix.  
\cite{DBLP:conf/cdc/Dominguez-GarciaH11a, DBLP:journals/tcns/SilvestreHS19} consider averaging algorithms for directed graphs where agents transmit their values to one or several agents, but they do not receive data. 
Nodes update their values using a weighted linear combination of their own value and the values of neighbouring nodes. 
In \cite{alistarh_et_al:LIPIcs:2020:12414} the authors study a dynamic load balancing process on cycles. 
In \cite{DBLP:conf/icalp/CaiS17}  assume that  the load inputs are drawn from a fixed probability distribution. In
\cite{DBLP:conf/podc/LunaB15, DBLP:conf/icalp/KowalskiM18, DBLP:conf/icalp/KowalskiM19} the authors use a diffusion algorithm for counting the nodes in an anonymous network and in \cite{DBLP:conf/focs/DotyEGSUS21} diffusion is used for a majority process in the population model. In \cite{DBLP:conf/stoc/FriedrichS09, DBLP:conf/focs/SauerwaldS12} the authors compare continuous with discrete diffusion processes. 

\section{Notation and Preliminaries} \label{sec:notation}
In order to prove detailed concentration bounds
(stated in \cref{thm:limitingVarExact1}),
we introduce additional notation and new concepts.
We denote our \avgprocess with $(\xi(t))_{t\ge 0}$,
where $\xi(t) = (\xi_u(t))_{u\in V}$. 
For any vector $\xi(t)\in \R^n$, we define the quantities
\begin{align}\label{eq:katze}
\avg(t)=\frac{1}{n}\sum_{u\in V}\xi_u(t) \qquad
\mbox{and}  \qquad
 M(t) =\sum_{u\in V}\frac{d_u}{2m}\cdot \xi_u(t).
 \end{align}
For any vector $v\in \R^n$, let $\| v \|_2 = \sqrt{\sum_i v_i^2 } $.
 Let $L=D-A$ be the graph Laplacian of $G$, where $D=\text{diag}(d_1,...,d_n)$  is the diagonal matrix with entries $d_i$ (degrees of nodes), and $A$ is the adjacency matrix of $G$. $L$ is a symmetric positive semi-definite matrix with eigenvalues
$0=\l_1(L)<\l_2(L) \le \cdots \le \l_n(L)$. Note that $\l_2(L) >0$ follows from the connectivity of $G$ ($L$ is irreducible). Further, $e^{(i)} \in \{ 0,1\}^n$ is the indicator (column) vector, where the $i$-th entry is $1$ and all other entries are $0$, 
$\one$ is a vector of all $1$'s (of appropriate length to fit the context),
and symbol $\top$ indicates the transpose of a vector or matrix 
(introduced to avoid confusion with the time step $T$). 

We consider the following {\em lazy variant}
of the $\mone$, where in each step with 
 probability $1/2$ the selected node performs no update, i.e.,
$\x(t)=\x(t-1)$ and otherwise behaves as before. This can be related to the transition matrix $P$ of a lazy random walk on $G$, with $p_{(i,i)} =1/2$ and $p_{(i,j)}=1/(2 d_i)$, for each $i\in V$ and $(i,j)\in E$. Let $\lambda_2(P)$ be the second-largest eigenvalue of $P$ and let $f_2(P)$ be the corresponding eigenvector. Let $\pi$ denote the vector with  $\pi_i$ equal to the probability that a fully mixed random walk is at the node $i$.  
We will use the $\pi$ weighted inner product of two vectors $\nu$ and $\nu'$ in $\R^n$
given by 
\begin{align}
\left \langle \nu,\nu' \right \rangle_{\pi} = \sum_{x\in V} \pi_x \nu_x \nu'_x.
\end{align}
Let $\pi_{max}=\D/2m$ and $\pi_{min}=\d/2m$.
Our goal is to analyse the time it takes until the nodes have almost identical values. To do so, we will make use of the following potential function.
\begin{align}
\phi(\xi(t)) &= \langle \xi(t), \xi(t) \rangle_\pi - \left \langle \one, \xi(t) \right \rangle_{\pi}^2 \nonumber\\
&= \frac{1}{2}\sum_{u,v\in V}\pi_u\pi_v(\xi_u(t)-\xi_v(t))^2.\label{eqn:defiPot}
\end{align}

We say that the process has $\epsilon$-\emph{converged} whenever $\phi(\xi(t)) \leq \epsilon$. 
For comparison with related notions:
a $(\epsilon/n)^6$-convergence   implies a  discrepancy $K$ (maximum value minus minimum value) of at most $\epsilon$.
Finally, a fundamental property of the processes defined in \cref{eq:katze} is that they are martingales, as shown in the following result (its proof is in the Appendix).
\begin{lemma}\label{lem:NodeMartingale}
We have
$\E(M(t+1)|\xi(t))= M(t).$
In particular,
$\E(M(t+1)|\xi(0))= M(0)$
and for regular graphs
$\E(M(t+1)|\xi(0)) = \avg(0).$
\end{lemma}

\section{Proof of the Concentration Bounds}\label{sec:concentration}
In this section we will show the second part of our main theorem which states results for the $\mone$. Note that this part of the theorem holds for regular graphs only, and recall that the results extend to the $\mtwo$, since for regular graphs both models are identical. 
To show our result we will study the dual $(W(t))_{t\ge0}$ of the \avgprocess $(\xi(t))_{\ge t}$, which we call the \dualprocess (for a detailed definition of the process see \cref{sec:dual}). The latter process  can be thought of as an $n$-dimensional diffusion process in which the dimensions are balanced separately but not  independently of each other. 
Due to the duality (which we show in \cref{prop:duality}) we get that the variance of the final node values of both processes is the same. 
Therefore, it is sufficient to study the variance of the dual process. 
In turn, to analyse the variance of the dual process we study the joint distribution $\mu$ of two Random Walk Processes $(\widetilde{W}(t))_{t\ge 0}$ (see \cref{sec:CorrelatedRWs} for a detailed definition), where  two random walks start in two (not necessarily different) nodes. 
We will present and prove the following relationships. 
%
%
\begin{equation*}
\begin{split}
\Var(M(t)) \;\; &\stackunder[5pt]{\approx}{\scalebox{0.8}{\color{cfirst}\cref{lem:dualitySupportLemma}}} \;\;
\Var(W(t)) \;\; \stackunder[5pt]{\approx}{\scalebox{0.8}{\color{cfirst}\cref{prop:variance_both_processes1}}} \Var(\widetilde{W}(t))\\
&\stackunder[5pt]{\approx}{\scalebox{0.8}{\color{cfirst}\cref{lem:crwTOstationarydistr}}} 
\;\; \sum_{u,v}\mu(u,v)\xi_u(0)\xi_v(0)
\end{split}
\end{equation*}
Through \cref{lem:dualitySupportLemma} and \cref{prop:variance_both_processes1} we show that the three processes have the same variance. Then, through \cref{lem:crwTOstationarydistr} we will show that this variance can be expressed in terms of $\mu$ and $\xi(0)$. Recall that we use $^\top$ to indicate the transpose of a vector.
\subsection{The \dualprocess}\label{sec:dual}
We introduce the \dualprocess, which is closely related to our \avgprocess. We call this relationship \emph{duality}.
The \dualprocess takes as parameters 
a cost (row) vector $c\in \mathbb{R}^n$
and an initial load (column) vector 
$q(0)\in \mathbb{R}^n$. 
At a time step $t\ge 0$, the load vector is 
denoted by $q(t)\in \mathbb{R}^n$ and its
cost is the value $c\, q(t)$.
At every time step $t\ge 1$, a node $u(t)$  is sampled uniformly at random.  Node $u(t)$ then chooses a subset $S(t)$ of $k$ neighbours uniformly at random and spreads a $1-\alpha$  fraction of its load uniformly to these $k$ neighbours. The choice of neighbours and load redistribution can be represented using matrix $B(t)$ defined as follows.
\begin{equation}\label{eq:pony}
B_{i,j}(t) = \begin{cases}
1, & \text{if $i=j\neq u(t)$,}\\
\alpha, & \text{if $i=j=u(t)$,}\\
(1-\alpha)/k, & \text{if $i
\in S(t)$ and $j=u(t)$,}\\
0, & \text{otherwise}.
\end{cases}
\end{equation}
Then at step $t\ge 1$, the load vector and 
the state of the process (the cost of the load
distribution) are given, respectively, by
\begin{align*}
q(t)&= B(t)\, q(t-1) =  R(t)\, q(0)
\;\; \mbox{and} \;\;
w(t) = c\, q(t) = c\, R(t)\, q(0),
\end{align*}
where
\begin{equation} \label{eq:llama}
R(t)= B(t)\cdot B(t-1)\cdots B(1).
\end{equation}

The \dualprocess can be applied simultaneously to 
a number of load vectors
$q^{(1)}, q^{(2)}, \ldots, q^{(r)}$,
$r \ge 1$, of different commodities:
\begin{align*}
(q^{(1)}(t), q^{(2)}(t), & \ldots,q^{(r)}(t)) \\
&= R(t)\, (q^{(1)}(0), q^{(2)}(0), 
\ldots, q^{(r)}(0)), \\
W(t) &= c\, R(t)\, (q^{(1)}(0),q^{(2)}(0),\ldots,q^{(r)}(0) ),
\end{align*}
where 
$q^{(j)}(t)$ and 
$W^{(j)}(t)$ ($j$-th entry in $W(t)$) are, respectively, the load vector~$j$
and its cost at step $t$.
Note that $W(t)  \in \mathbb{R}^r$.
%
%
%
%
%
%
Next, we formally state the duality between the \avgprocess and the \dualprocess.
For $u\in V$, the column vector $e^{(u)}$ is
the unit vector with $1$ at position~$u$
and $0$'s everywhere else.

\begin{proposition}\label{prop:duality}
If the \dualprocess is applied to cost vector 
$c = \xi^\top(0)$
and to $n$ initial load vectors $e^{(u)}$, $u\in V$
(that is, the total load -- one unit -- of `commodity' $u$ is initially at node $u$),
then for each $T\ge 0$, the probability 
distribution of $\xi(T)$ in the \avgprocess
is the same as
the probability distribution of 
$W(T)$ in the \dualprocess.
That is,
for any (column) vector $a\in \mathbb{R}^n$,
$\Prob(\xi(T) = a) = \Prob(W(T) = a^\top)$.
\end{proposition}

\begin{proof}
This duality relation follows by coupling the
\avgprocess with the \dualprocess, running one
of them (say the \avgprocess) forward in time 
and the other backwards in time.
More formally, we fix an arbitrary time step
$T\ge 1$ and consider any feasible 
node selection sequence
$\newChi=(\newChi(1), \newChi(2), \dots, \newChi(T))$ for $T$ steps.
That is, for $1\le t \le T$, 
$\newChi(t) = (S(t)), u(t))$, 
where $u(t)$ and $S(t)$ are a node
and a size $k$ sample of its neighbours.
The lemma below shows that if
we run the \avgprocess using sequence $\newChi$  
and the \dualprocess using the reverse sequence 
$\newChi^R$, then $W(T)=\xi^\top(T)$.
The proposition then follows
because the probability of having 
sequences $\newChi$ in the first $T$ steps 
in the \avgprocess is the same as 
the probability of having sequence $\newChi^R$
in the first $T$ steps in the \dualprocess.

Note that running one of the two processes backward
is crucial for establishing this duality. 
If both processes are run forward on the same 
sequence $\newChi$, then most likely 
$W(T) \neq \xi^\top(T)$.
\end{proof}
\begin{lemma}\label{lem:dualitySupportLemma}
For the \avgprocess and the \dualprocess as in
\cref{prop:duality}, 
an arbitrary time step $T\ge 1$, and 
an arbitrary node selection sequence 
$\newChi=(\newChi(1), \newChi(2),$ $ \dots, \newChi(T))$, if we run the \avgprocess using sequence $\newChi$
and the \dualprocess using the reverse sequence 
$\newChi^R$, then we have $W(T)=\xi^\top(T)$.
\end{lemma}

\begin{proof}
Since for each $1 \le t \le T$, in step $t$ the \avgprocess uses the same
node selection as in the \dualprocess
in step $T+1-t$. Then
$\xi(t) = F(t) \xi(t-1)$, where
$F(t) = B^\top(T-t)$ and thus we have
\begin{align*}
W(T)\; &= \; c\cdot\left(B(T) B(T-1)\ldots B(1)\right)\cdot
\left(e^{(1)}, e^{(2)}, \ldots, e^{(n)}\right)\\
&= \; c\cdot \left(B(T) B(T-1) \ldots B(1)\right)\cdot I\\
&= \; \xi^\top(0)\cdot F^\top(1) F^\top(2) \ldots F^\top(T)\\
&= \; \left(F(T) \ldots F(2) F(1)\cdot \xi(0)\right)^\top \;\; = \; \xi^\top(T).
\end{align*}
\end{proof}

From now on, \dualprocess means 
the diffusion process as specified in 
\cref{prop:duality}.
For an illustrative example of how the \avgprocess and the \dualprocess work and relate to each other, see \cref{fig:coupling}. In there, we fix $\alpha = 1/2$ and $k=1$. For another example with $\alpha=1/2$ and $k> 1$ see \cref{fig:couplingk2}.
%
\begin{figure*}
\centering
\begin{minipage}[t]{.49\textwidth}
    \includegraphics[width=.76\textwidth]{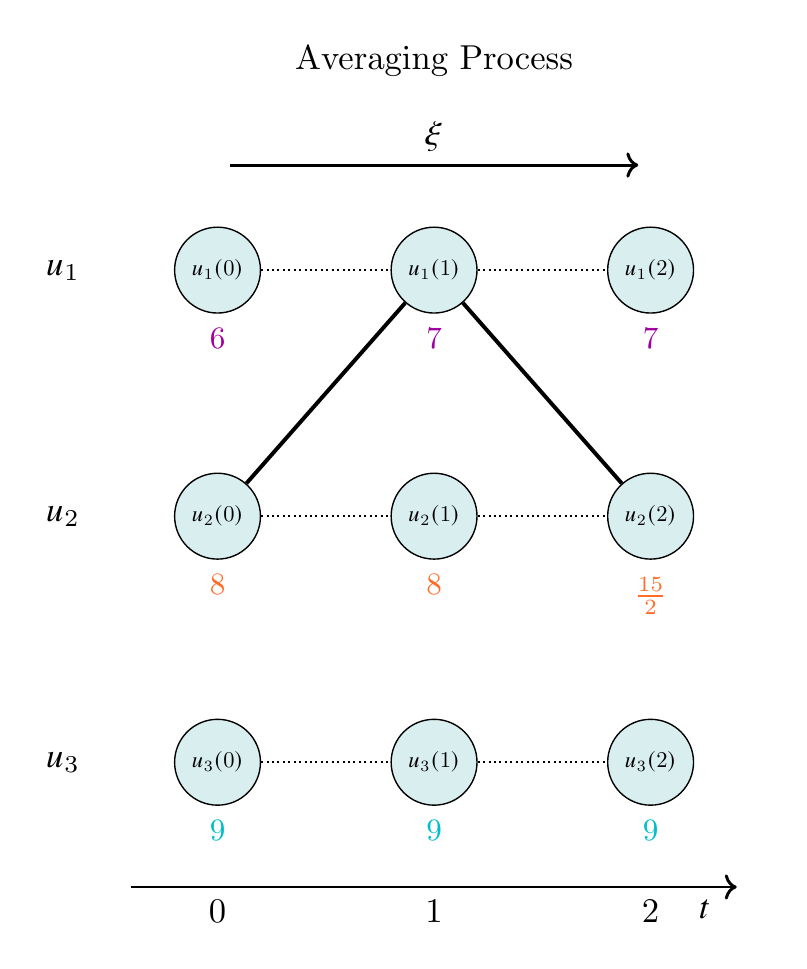}
    \caption*{(a)}
\end{minipage}
\begin{minipage}[t]{.49\textwidth}
    \includegraphics[width=.76\textwidth]{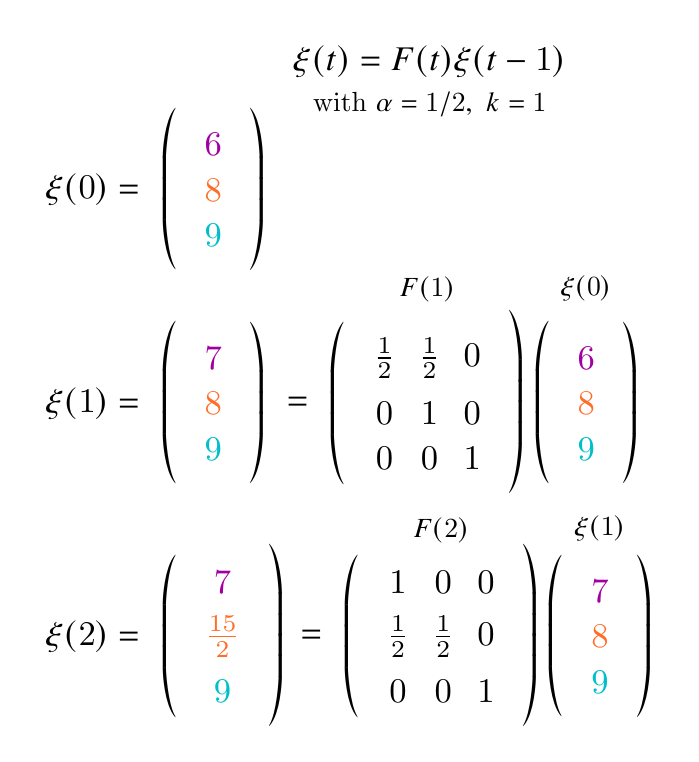}
\end{minipage}
\begin{minipage}[t]{.49\textwidth}
    \includegraphics[width=.76\textwidth]{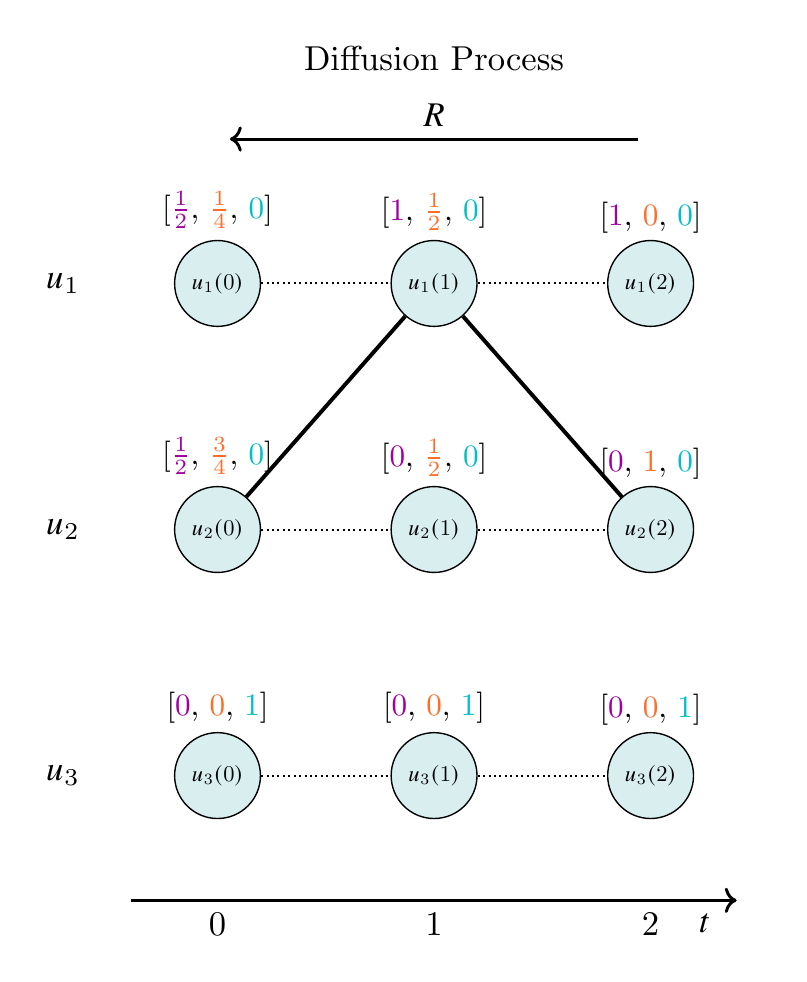}
    \caption*{(b)}
\end{minipage}
\begin{minipage}[t]{.49\textwidth}
    \includegraphics[width=.76\textwidth]{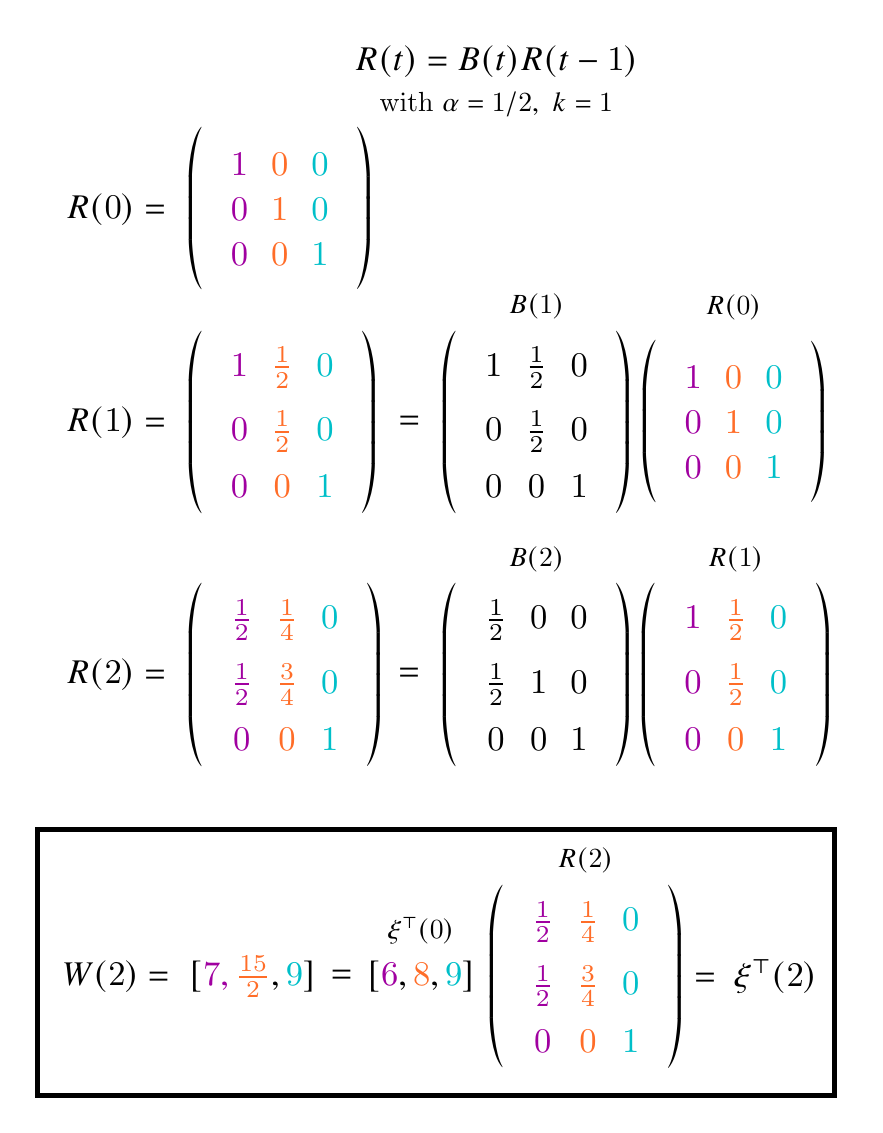}
\end{minipage}
\caption{Illustration of the duality between the \avgprocess (a) and 
the \dualprocess (b), with $k=1$ and $\alpha=1/2$. 
In (a), at $t=1$ node $u_1$ and its neighbour $u_2$ are selected,
the value at $u_1$ is updated and 
the values at $u_2$ and $u_3$ stay the same -- see matrix $F(1)$. At $t=2$, $u_2$ and its neighbour $u_1$ are selected, leading to $\xi(2)$
-- see matrix $F(2)$.
In (b), the \dualprocess runs backwards. The initial state, at $t=2$, of the diffusion starting in $u_2$ is the vector $[{\color{csecond}0,1,0}]$. 
At the first step ($t=2$), $u_2$ sends 
$1-\alpha = 1/2$ of its load to $u_1$. 
The loads in the other nodes do not spread. The resulting load vector is $R_2(1) = [{\color{csecond}1/2, 1/2, 0}]$, the second column of $R(1)$.
After the second step, the load is $R_2(2) = [{\color{csecond}1/4, 3/4, 0}]$. The diffusion of the loads originating 
at nodes $u_1$ and $u_3$ is indicated in {\color{cfirst}\cfirst} and {\color{cthird}\cthird} respectively.
We get $W(2) = \xi^\top(2)$.}
\label{fig:coupling}
\end{figure*}
\subsection{The \crwprocess}\label{sec:CorrelatedRWs}
%
%
\sloppy With the \dualprocess we associate $n$ random walks 
$\left(\tilde{q}^{(u)}(t)\right)_{t\ge 0}$, for $u\in V$.
At step $t$, 
the position of walk $u$ is
$\tilde{q}^{(u)}(t)\in \{e^{(1)}, e^{(2)}, \dots, e^{(n)}\}$, where the `$1$' in vector $\tilde{q}^{(u)}(t)$
indicates the node where the walk is at 
this step. 
The cost of this walk at step $t$ is defined as
$\widetilde{W}^{(u)}(t)=\xi(0) \cdot \tilde{q}^{(u)}(t)$, from the \avgprocess.
Thus, if the walk is at a node $v\in V$,
its cost is $\widetilde{W}^{(u)}(t) = \xi_v(0)$.
The starting node of this walk is node $u$, 
that is, $\tilde{q}^{(u)}(0) = e^{(u)}$.
%
 %
 At step $t$, the transition matrix for all $n$ random walks is matrix $B(t)$ of the \dualprocess given in \cref{eq:pony}. 
 That is, given the location 
 of random walk $u$ at step $t-1$,
 represented by vector 
 $\tilde{q}^{(u)}(t-1)$, and fixing 
 the transition matrix $B(t)$, the distribution of the location of this walk 
 at step $t$
 is given by 
 $B(t) \tilde{q}^{(u)}(t-1) = R(t) e^{(u)}$.
Hence, if $B(t)$ represents choosing 
node $v$ and its $k$ neighbours $w_1, w_2, \ldots, w_k$, and if the walk happens to be at node $v$ at step $t-1$, then in step $t$ the walk moves to node~$w_i$ with probability $(1-\alpha)/k$,
for $i=1,2,\ldots,k$.
Otherwise, the walk does not move in step $t$.
In what follows, `the random walk starting at node $u$' means the random walk defined here.
%
%
The relation between the random walks defined above and the \dualprocess is akin to the relation between the standard random walk and the standard diffusion process.\footnote{In the standard random walk process the walk moves to all neighbors with equal probability and in the standard
diffusion process, each node sends the same fraction of load to all its neighbors synchronously starting with load one at one node and load zero elsewhere.} 
The former describes the  location of the random walk at a given time (the walk can only be at one position at any time), while
the latter gives the distribution of the random walk at any given time step. 
\begin{lemma}\label{lem:crwVSdiffusion}
For any fixed sequence of transition metrices
$\newChi = (B(1), B(2), \ldots, B(t))$,
for the random walk starting at node $u$, we have (recall the definition if $R(t)$ in~\eqref{eq:llama}),
\begin{equation}\label{eq:relationship}
\E{[\tilde{q}^{(u)}(t)\; |\; \newChi]} =R(t) e^{(u)},
\end{equation}
\begin{equation}\label{eq:relationship2}
\E{[\widetilde{W}^{(u)}(t)\; |\; \newChi]} = W^{(u)}(t)  , 
\end{equation}
\end{lemma}
\label{comm:dimensions}
\begin{proof}
Given $\newChi$, 
the vector $R(t) e^{(u)}$ is the distribution at step $t$ of the random walk starting at node $u$.
This means that for $x\in V$, the expectation of 
$\tilde{q}^{(u)}_x(t) \in \{0,1\}$ is equal to entry $x$ 
in $R(t) e^{(u)}$, which gives the 
probability that at step $t$ the walk is
at node $x$ 
(the probability that $\tilde{q}^{(u)}_x(t) =1$).
Hence \cref{eq:relationship}.\footnote{%
$\E[(X_1,X_2, \ldots, X_n)]$ is defined as 
$(\E[X_1], \E[X_2], \ldots, \E[X_n])$.}
For \cref{eq:relationship2}, by linearity of expectation,
\begin{align*}
    \E{[\widetilde{W}^{(u)}(t)\; |\; \newChi]} 
    &= \E{[\xi^\top \tilde{q}^{(u)}(t)\; |\; \newChi]} \\
    &= \xi^\top \cdot\E{[\tilde{q}^{(u)}(t)\; |\; \newChi]}
    = \xi^\top \, R(t) \, e^{(u)}\\
    &= W(t) e^{(u)} = W^{(u)}(t).
\end{align*}
where $\xi = \xi(0)$.
\end{proof}

We refer collectively to the $n$ random walks defined in this section as \crwprocess
$\widetilde{W}(t)$.
These $n$ random walks are correlated since 
they use the same random choices of nodes and neighbours sampled, that is, the same transition matrices $B(t)$.
\begin{proposition}\label{prop:variance_both_processes1} For the \dualprocess $W(t)$ and the \crwprocess $\widetilde{W}(t)$, we have
\begin{align}
\E[\widetilde{W}^{(u)}(t) \widetilde{W}^{(v)}(t)] = \E[W^{(u)}(t) W^{(v)}(t)]
\end{align}
\end{proposition}
\begin{proof}
%
It is sufficient to prove
that for any fixed sequence $\newChi$ of 
transition matrices,
\begin{align}
\E[\widetilde{W}^{(u)}(t) \widetilde{W}^{(v)}(t)\,|\,\newChi] = W^{(u)}(t) W^{(v)}(t).
\end{align}
Note that if $\newChi$ is fixed, then there is no randomness on the right-hand side of the above. By writing $\xi$ instead of $\xi(0)$, we have

\begin{align}
\E[\widetilde{W}^{(u)}(t) \widetilde{W}^{(v)}(t)\,|\,\newChi] &= \E[ \xi^\top \tilde q^{(u)}(t) \cdot \tilde{q}^{(v)\top}(t) \xi \,|\,\newChi] \nonumber\\
&=\xi^\top \cdot \E[  \tilde q^{(u)}(t) \cdot  \tilde{q}^{(v)\top}(t) \, |\,\newChi] \cdot \xi \label{njkewecni-1} \\ 
&= 
\xi^\top \cdot\E[ \tilde q^{(u)}(t)\,|\,\newChi] \cdot \E[ \tilde{q}^{(v)\top}(t) \,|\,\newChi] \cdot \xi
\label{njkewecni-2} \\
&= 
\left(\xi^\top  R(t) \indi{u}\right) \cdot \left(\indi{v}^\top R^\top(t) \xi\right)
\label{njkewecni-3} \\
&={W^{(u)}}(t) \cdot W^{(v)}(t),\nonumber
\end{align}
where~\eqref{njkewecni-1} is by linearity of expectation,
\eqref{njkewecni-2} is by independence of the random walks 
$\tilde q^{(u)}(t)$ and $\tilde q^{(v)}(t)$
(once the sequence $\newChi$ is fixed), and
\eqref{njkewecni-3} follows from \cref{eq:relationship}.
%
%
\end{proof}
\subsection{Joint Distribution of Two Random Walks}\label{sec:mu}

In this section we consider two correlated random walks $\tilde q^{(a)}(t)$ and 
$\tilde q^{(b)}(t)$, $a,b\in V$,
$a\neq b$. Respectively, we denote by $X(t)$ and $Y(t)$ the nodes where these walks are at step $t$.
The random walks proceed through the \crwprocess described in the previous section, resulting in a joint transition matrix $Q$:
\begin{eqnarray*}
    \lefteqn{Q((x,y),(u,v))} \\
    & = & \Prob((X(t+1),Y(t+1))=(u,v)\;|\;(X(t),Y(t))=(x,y)),
\end{eqnarray*}
where $(x,y), (u,v) \in V\times V$.
The $Q$ chain defined by this transition matrix 
is irreducible (each state $(u,v)$ is reachable form each state $(x,y)$) and aperiodic 
($Q(s,s) > 0$, for each $s\in V\times V$),
so it has a unique stationary distribution.
We use $\pm f(n)$ to denote a term $cf(n)$ 
where $|c| \le 1$.
\begin{lemma}\label{lem:crwTOstationarydistr}
For the stationary distribution $\mu$ of the $Q$ chain defined above and sufficiently large $T$, we have, 
\begin{equation}\label{eq:mixingTime}
    \E[\widetilde{W}^{(a)}(T) \widetilde{W}^{(b)}(T)] = \sum_{u,v}\mu(u,v)\xi_u(0)\xi_v(0) \; \pm \; \frac{1}{n^5}.
\end{equation}
\end{lemma}
\begin{proof}
%
Let $T$ be the mixing time of $Q$ such that

$\left|\mu(u,v) - Q^T((a,b), (u,v))\right| \leq \frac{1}{K^2n^7}$, for each $(u,v)\in V\times V$,
where $K$ is the initial discrepancy. 
%
Now, let's write $\xi$ instead of $\xi(0)$, then
\begin{align*}
\E\left[\,\widetilde{W}^{(a)}(T)\right. & \left.\widetilde{W}^{(b)}(T)\,\right] \\
&\asterisk
\E\left[\xi^\top\, \tilde q^{(a)}(T)\, \left(\tilde{q}^{(b)}\right)^\top(T)\, \xi \right] \\ 
 &= \sum_{u,v} 
 \Pr\left(\tilde q^{(a)}(T)=\indi{u},\,
 \tilde q^{(b)}(T)=\indi{v}\right)\xi_u\xi_v\\
 &=\sum_{u, v} Q^T((a,b),(u,v))\xi_u\xi_v\\
 &=\sum_{u,v}\left(\mu(u,v)\pm \frac{1}{K^2n^7}\right)\xi_u\xi_v\\
&=\sum_{u,v}\mu(u,v)\xi_u\xi_v \; \pm \; \frac{1}{n^5},
\end{align*}
where $(^*)$ follows from the definition of \crwprocess.

%
%
\end{proof}
We calculate the entries of matrix $Q$ and
find its exact stationary distribution.
%
%
From our setting, there are 3 different types of transitions: neither walk leaves its current node (a self-loop), only one walk moves, or both walks move. 
If the walks are in the same node, then 
they can both travel to the same node or 
to two different nodes (or one or both could stay in their current node). The transition where both walks are moving requires that they be in the same node.
\allowdisplaybreaks
\paragraph{Case1: Both walks are at the same node $x$} Then, for some nodes $u$ and $v$ s.t. 
$x\neq u \neq v \neq x$,
\begin{align}
Q((x,x),(u,v))&=  (1-\alpha)^2 \pi_x  \frac{k}{d}\frac{k-1}{d -1}\frac{1}{k^2} = (1-\alpha)^2 \pi_x  \frac{k-1}{kd(d-1)},
\label{eq:matrixQ2_1}\\
Q((x,x),(u,u))&=  (1-\alpha)^2 \pi_x  \frac{k}{d}\frac{1}{k^2} = (1-\alpha)^2 \pi_x  \frac{1}{kd}, \label{eq:matrixQ2_2}\\
Q((x,x),(x,u))&= \alpha(1-\alpha) \pi_x \frac{1}{d}, \label{eq:matrixQ2_3}\\
Q((x,x),(u,x))&= \alpha(1-\alpha) \pi_x \frac{1}{d}, \label{eq:matrixQ2_4}\\
Q((x,x),(x,x))&= \alpha^2 \pi_x  + (1-\pi_x).\label{eq:matrixQ2_5}
\end{align}
We explain the meaning of \cref{eq:matrixQ2_1}, the other equations follow in a similar manner. If both walks are in $x$, then for them to have a chance of moving, we need to sample node $x$ first. This happens with probability $\pi_x$. 
The probability of both walks moving away from $x$ is equal to $(1 -\alpha)^2$. 
The first walk goes to $u$ and the second to $v$, if both $u$ and $v$ are in the selected $k$-sample of neighbours of $x$ -- probability $\frac{k(k-1)}{d(d-1)}$ -- and then 
the first walk chooses $u$ and the second one chooses $v$ -- probability~$\frac{1}{k^2}$.
\paragraph{Case2: The walks are on two different nodes $x \neq y$} Then, for a node $v \neq y$ and a node $u \neq x$,
\begin{align}
Q((x,y),(x,v))&= (1-\alpha) \pi_y \frac{1}{d} \label{eq:matrixQ3_1} \\
Q((x,y),(u,y))&= (1-\alpha) \pi_x \frac{1}{d} \label{eq:matrixQ3_2}\\
Q((x,y),(x,y))&= (1-\pi_x - \pi_y)  + (\pi_x + \pi_y)\alpha \label{eq:matrixQ3_3}
\end{align}
All other transition probabilities are 0. The cases above hold for any graph. In the case of a d-regular graph, then we have that $P(x, u) = \frac{1}{d} \ \forall u \in N(x) $.
Note that $Q$ is not reversible, however a stationary distribution is still possible to find, since the chain is irreducible and positive recurrent. For an example explaining why $Q$ is not reversible, see the proof of \cref{prop:stationaryDistrib}.
%
\extend{In the case of the $\mtwo$, the elements of $Q$ are slightly different. There, we are looking at general graphs (hence, even irregular ones), meaning that $1/d$ would turn into $1/d_x$ for any node $x \in V$. Moreover, the probability of sampling a node $x$, $\pi_x$, would change from $1/n$ to $d_x/(2m)$.}

Next, we define three sets of different types of states of the $Q$-chain and give the formula
for the stationary distribution of this chain.

\begin{definition} \label{def:threeSets}
Consider a $d$-regular graph $G=(V,E)$ on $n$ nodes, and for $u,v\in V$, 
let $dis(u,v)$ denote the length of a shortest path between $u$ and $v$ in $G$.
For $i \ge 0$, we define sets
$S_i=\{ (u,v) \mid u,v \in V, dis(u,v)=i \}$
as subsets of the state space $V\times V$ of the $Q$-chain. 
We also define $S_+ = \bigcup_{i \geq 2} S_i$.
\end{definition}


\begin{lemma} \label{prop:stationaryDistrib}
Consider the $\mone$ with parameters $\alpha\in (0,1)$ and $k\ge 1$.
The stationary distribution $\mu$ of the $Q$-chain is comprised of exactly three different values: 
for a state $(u,v) \in V\times V$, 
%
\begin{align} \label{eq:threeSets}
\mu(u,v) = \begin{cases}
\mu_0(n,d,k,\alpha)=2k(d-1)\ell & \text{if } (u,v) \in S_0\\
\mu_1(n,d,k,\alpha)=(d-1)\gamma \ell & \text{if } (u,v) \in S_1  \\
\mu_+(n,d,k,\alpha)=(d\gamma-2\alpha k)\ell& \text{if } (u,v) \in S_+
\end{cases}
\end{align}
with $\gamma = k(1+\alpha) - (1-\alpha)$, and $\ell = \frac{1}{{n (n(d\gamma-2\alpha k) + 2(1-\alpha)(d-k))}}$.

\end{lemma}
All missing proofs can be found in the Appendix.
\subsection{Proof of \texorpdfstring{\cref{thm:mone}(2)}{}: Concentration of the Convergence Value}
\balance
We present here the following proposition, which we use to show the second part of \cref{thm:mone}.
\begin{proposition}\label{thm:limitingVarExact1}
Consider $\mone$ with parameters $\alpha\in (0,1)$ and $k\ge 1$. W.l.o.g.  we assume that  $\avg(0)=0$.
Let $E^+ = \{ (u,v)  | \{u,v\}\in E\}$ be the set of directed edges in the underlying graph.
Then, for any $t\ge 0$,
\begin{align}\label{eq:concentrationFirst1}
 &\var(\avg(t)) \leq \notag\\
%
&\leq(\mu_0-\mu_+) \sum_{u \in V} \xi_u^2(0) + (\mu_1-\mu_+) \sum_{(u,v) \in E^+} \xi_u(0)\xi_v(0) + 1/n^5,
\end{align}
Furthermore, there exist a $T$ large enough such that for all $t\geq T$
\begin{align}\label{eq:concentrationFirst2}
 &\var(\avg(t)) \geq \notag\\
%
&\geq(\mu_0-\mu_+) \sum_{u \in V} \xi_u^2(0) + (\mu_1-\mu_+) \sum_{(u,v) \in E^+} \xi_u(0)\xi_v(0) - 1/n^5,
\end{align}

\end{proposition}

\begin{proof}
The random variable $\avg(t) = \frac{1}{n}\sum_{x\in V} \xi_x(t)$ is a Martingale (\cref{lem:NodeMartingale}).
Due to our convergence results, we know that as $t \rightarrow \infty$, all node values 
$\xi_x(t)$ converge to the same value. 
Consequently, $\avg(t)$ converges 
to the same value, which we denote by 
$\avg(\infty)$.
This is a random variable 
with expectation $\E[\avg(\infty)] = \avg(0)$, 
since by the Martingale property, for each $t\ge 0$, 
$\E[\avg(t)] = \avg(0)$.
We want to show that the actual value of $\avg(\infty)$ is likely to be close to $\avg(0)$. We do this by showing that $\var(\avg(\infty)) = \lim_{t \rightarrow \infty} \var(\avg(t))$ is small.
We start by 
recalling the assumption that 
$\avg(0)=0$ and using linearity of expectation to obtain
\begin{align*}
\var(\avg(t)) &= \E[\avg(t)^2] - (\E[\avg(t])^2
\\
&= \E[\avg(t)^2]\\
&= \frac{1}{n^2}
\sum_{x,y\in V} \E[\xi_x(t)\xi_y(t)].
\end{align*}

Fixing $t=T$ arbitrarily, 
Propositions~\ref{prop:duality} 
and~\ref{prop:variance_both_processes1} imply that for any pair of nodes $x$ and $y$,
the three products
$\xi_x(T) \xi_y(T)$, $W^{(x)}(T) W^{(y)}(T)$ and $\widetilde{W}^{(x)}(T) \widetilde{W}^{(y)}(T)$ have the same expectation, so 
\begin{align*}
    \var(\avg(T)) &= \frac{1}{n^2}
    \sum_{x,y} \E[\xi_x(T)\xi_y(T)]\\
    &=\frac{1}{n^2}\sum_{x,y} \E[W^{(x)}(T)W^{(y)}(T)]
    \\
    &=\frac{1}{n^2}\sum_{x,y} \E[\widetilde{W}^{(x)}(T)\widetilde{W}^{(y)}(T)]. 
\end{align*}
%
%
Let $\xi = \xi(0)$, then by \cref{lem:crwTOstationarydistr}
and \cref{prop:stationaryDistrib} we get
\begin{align*}
    \var(\avg(T)) &= 
    \frac{1}{n^2}\sum_{x,y} \left(\sum_{u,v}\mu(u,v) \xi_u\xi_v \pm 1/n^5\right) 
    =  \sum_{u,v}\mu(u,v) \xi_u\xi_v \pm 1/n^5\\
    &= \mu_0 \sum_{(u,u) \in S_0} \xi_u^2\ + \mu_1 \sum_{(u,v) \in S_1} \xi_u\xi_v    +\mu_+ \sum_{(u,v) \in S_+} \xi_u\xi_v \pm 1/n^5\\
    & =  (\mu_0-\mu_+) \sum_{(u,u) \in S_0} \xi_u^2 + (\mu_1-\mu_+) \sum_{(u,v) \in S_1} \xi_u\xi_v +\\
    &\qquad \; \; \quad +\mu_+ \sum_{u,v \in V} \xi_u\xi_v \pm 1/n^5\\
    &= (\mu_0-\mu_+) \sum_{u \in V} \xi_u^2 + (\mu_1-\mu_+) \sum_{(u,v) \in E} \xi_u^2 \pm 1/n^5, 
\end{align*}
using for the last equation that
$\sum_{u,v}\xi_u\xi_v=
(\sum_{u}\xi_u)(\sum_u\xi_u)= 0$
(the assumption that $\avg(0)=0$).
Thus \cref{eq:concentrationFirst2} holds.
For \cref{eq:concentrationFirst2}, observe 
that $\var(\avg(t))$ is non-decreasing.
\end{proof}

Note that our techniques only allow us to derive bounds on the variance that come from the mixing time of the Q-chain. Hence, we do not have tight bounds on the variance at the beginning of the process. However, at the end of the process (after the mixing time) our results are asymptotically tight. 
We are now ready to prove \cref{thm:mone}(2).
\begin{proof}
Let $\xi=\xi(0)$, and, 
%
%
%
referring to the right-hand sides in
\eqref{eq:concentrationFirst1} and \eqref{eq:concentrationFirst2},
we derive
 \begin{align}
 & (\mu_0-\mu_+)\hspace{-0.1cm} \sum_{u \in V} \xi_u^2 \; + \; (\mu_1-\mu_+)\hspace{-0.2cm} \sum_{(u,v) \in E^+} \xi_u\xi_v 
 =\nonumber\\
 &\;\;\;\;=((\mu_0-\mu_+)-d(\mu_1-\mu_+)) \twonorm{\xi} \notag \\
 &\phantom{Xxxxxxxxxxx}+ (\mu_1-\mu_+)\left(\sum_{(u,v) \in E^+} \xi_u\xi_v + d \sum_{u\in V} \xi_u^2 \right) \label{jkje5s}
 \end{align}
Observe that 
$
0 \le \sum_{(u,v) \in E^+} \xi_u\xi_v + d \sum_{u\in V} \xi_u^2 \le 2d\twonorm{\xi},
$
which follows from
$
\sum_{(u,v) \in E^+} \xi_u\xi_v + d \sum_{u\in V} \xi_u^2 \; = 
\sum_{\{u,v\} \in E} (\xi_u+\xi_v)^2
\; \le \sum_{\{u,v\} \in E} 2(\xi_u^2+\xi_v^2)
\: =\: 2d \sum_{u\in V} \xi_u^2
\: =\: 2d\twonorm{\xi}.
$
Then, by noting that $\mu_1 - \mu_+ \leq 0$, we have from \cref{eq:concentrationFirst1} and \cref{jkje5s} that
$\var(F) - 1/n^5 \leq  ((\mu_0-\mu_+)-d(\mu_1-\mu_+)) \twonorm{\xi} =\frac{2k(d-1)(1-\alpha)}{n^2(3dk+d-3k)} \twonorm{\xi} = O\left( \frac{\twonorm{\xi}}{n^2}\right).
$
From \cref{eq:concentrationFirst2} and \cref{jkje5s} we obtain
 $\var(F) + 1/n^5 \geq  \frac{2k(d-1)(1-\alpha)}{n^2(3dk+d-3k)} \twonorm{\xi} 
 + \frac{-k(1+\alpha)+(1-\alpha)+2\alpha k}{n^2(3dk+d-3k)}2 d\twonorm{\xi}= \frac{2(1-\alpha)(2dk-d-k)}{n^2(3dk+d-3k)} \twonorm{\xi}= \Omega\left( \frac{\twonorm{\xi}}{n^2}\right).
$
\end{proof}

 \section{Future Work}
 There are two enticing lines of future work. First,  we have shown how to obtain tight bounds on $\var(F)$ by analysing the distribution of two dependent random walks. Can one obtain bounds on higher moments $M>2$ by considering $M$-dependent random walks? This would allow to derive stronger Chernoff-type results for the concentration of $F$. 
Second, is it possible to bound the concentration in the $\mone$ and $\mtwo$ for irregular graphs?

\section*{Acknowledgments}
P. Berenbrink was in part supported by the Deutsche Forschungsgemeinschaft (DFG) - Project number 491453517 and number DFG-FOR 2975.
F. Mallmann-Trenn was in part supported by the EPSRC grant EP/W005573/1.

\bibliographystyle{acm}
\bibliography{biblio}

\newpage

\onecolumn
\tableofcontents
\section*{Appendix}
\appendix
\section{Martingale Property - Proof of 
Lemma~\ref{lem:NodeMartingale} 
} \label{app:proofMartingale}
We show here that $M(t)$ is a martingale, proving \cref{lem:NodeMartingale}


\begin{proof}
Let $I$ denotes the identity matrix of size $n \times n$, and $J_{ij}$ the matrix that has exactly the value 1 in coordinate row i, column j, and 0 otherwise.

Suppose the current values of the nodes are $\xi$, and let $\xi'$ the values after one iteration. Note that the evolution of the averaging process can then be written as
$$\xi' = \left(\alpha I + \frac{1-\alpha}{k} \sum_{i=1}^k I + J_{u,v_i}-J_{u,u}\right)\xi$$
where $u\in V$ is the chosen node to update its value, and $(v_i)$ are the neighbours selected by $u$. Let $L =  \left(\alpha I + \frac{1-\alpha}{k} \sum_{i=1}^k I + J_{u,v_i}-J_{u,u}\right)$ the (random) matrix that gives the value of $\xi'$.

Note that $M(t) = \langle \xi, \mathbf{1}\rangle_{\pi}$ (where $\mathbf{1}$ is the vector with one 1's), and thus $M(t+1) = \langle L\xi, \mathbf{1}\rangle_{\pi}$. By linearity of expectation and the fact that  $L$ is independent of the current values $\xi$ we have
\begin{align*}
    \E(M(t+1)|\xi) = \langle \E(L)\xi, \mathbf{1}\rangle_{\pi}.
\end{align*}
We shall compute $\E(L)$. By linearity again, we have $\E(L) = \alpha I + (1-\alpha)I+\E(J_{u,v_i})-\E(J_{u,u})$.

It is easy to see that $\E(J_{u,u}) = \frac{1}{n}\sum_{w\in V} J_{w,w} = \frac{1}{n}I$. Now, conditioned on $u$ we have that
$$\E(J_{u,v_i}|u) = \sum_{w\in V} \ind{u\sim w} \frac{1}{d(u)} J_{u,w}$$
and thus
$$\E(J_{u,v_i}) = \frac{1}{n}\sum_{z\in V}\sum_{w\in V} \ind{z\sim w} \frac{1}{d(z)} J_{z,w}.$$
Then, if we ask for coordinate $(i,j)$ of the previous matrix we find it is equal to $\frac{1}{nd(i)}$ if $i\sim j$, 0 otherwise. I.e. the matrix is $\frac{1}{n}P$, where $P$ is the transition matrix of the random walk on $G$. Therefore $\E(L) = (1-1/n) I+ (1/n)P$, and thus

\begin{align*}
    \E(M(t+1)|\xi) &= \langle ((1-1/n) I+ (1/n)P)\xi, \mathbf{1}\rangle_{\pi}= \langle \xi, ((1-1/n) I+ (1/n)P)\mathbf{1}\rangle_{\pi}= \langle \xi, \mathbf{1}\rangle_{\pi}=M(t).
\end{align*}
The second equality holds since $I$ and $P$ are both self-adjoint with respect to $\langle \cdot, \cdot \rangle_{\pi}$, and $P\mathbf{1} = \mathbf{1}$.

\end{proof}
\section{\mone -- Convergence Time} \label{app:nodeConv}
\label{sec:TimeConvergenceFramework}

Recall that at each time step $t\geq 0$,  a node $u$ is chosen uniformly at random. That node, in turn,  samples $k$ neighbours $v_1,\ldots, v_k$ uniformly at random without replacement. 
The value of $u$ changes to 
\begin{align}\label{eqn:changexi}
\xi_{u}(t+1) = \alpha \xi_{u}(t)+(1-\alpha)\frac{1}{k}\sum_{i=1}^k\xi_{v_i}(t) 
\end{align}
and all other nodes keep their previous values. 
Recall that  the with $\pi$ weighted inner product of two vectors $\nu$ and $\nu'$ in $\R^n$
is denoted by 
$
\left \langle \nu,\nu' \right \rangle_{\pi} = \sum_{x\in V} \pi_x \nu_x \nu'_x.
$
and recall that $\|\xi\|_{\pi}^2 = \left \langle \xi, \xi \right \rangle_{\pi}$.
Note that for regular graphs,
\begin{equation}\label{eq:pony2}
 \phi(\xi(0))= \frac{1}{2n^2}\sum_{u,v \in V} (\xi_u(0)-\xi_v(0))^2=\frac{1}{n}\sum_{u \in V} (\xi_u(0)-\avg(0))^2=\frac{\twonorm{\xi(0)}}{n}.
\end{equation}
In general we have,
\begin{equation}\label{eq:ponies}
\pi_{min}^2 n\twonorm{\xi(0)} \leq \phi(\xi(0))\leq \pi_{max}^2 n\twonorm{\xi(0)}.
\end{equation}

We continue by introducing some additional notation.
Recall that $P$ is the transition matrix of the simple, lazy random walk on $G$ whose entries are $P(x,y)$. Further, that $\pi$ is the stationary distribution of $P$ and note that all vectors are in $\R^V$, and matrices are of dimension $\R^{V\times V}$.
Throughout this proof we use capital letters $X$ and $Y$ to  represent random variables accounting for nodes. Lowercase letters $x$ and $y$ represent fixed outcomes for these random variables. 
With  $\one$ we denote the vector where all entries are $1$, $I$ denotes the identity matrix and $\delta_{xy}$ is the matrix with a $1$ in position $x,y$, and $0$ elsewhere. We have \begin{align}\label{eq:hut}
\left (\delta_{xy}\cdot \xi(t)\right)_z  = 
\begin{cases} 
          \xi_y(t) & \text{if }z = x \\
          0 & \text{if }z\neq x
\end{cases}
\end{align}
With this notation, we have that \cref{eqn:changexi} can be written as
\begin{align}\label{eqn:changexiDelta}
\xi(t+1) = \left(I+ \frac{(1-\alpha)}{k}\sum_{i=1}^k (\delta_{XY_i}-\delta_{XX})\right)\xi(t)
\end{align}
where $X$ is sampled uniformly at random and $Y_i$ are (independent) random neighbours of $X$.

We now proceed to find the upper and lower bounds in \cref{app:upper} and \cref{app:lower}.

\subsection{Upper Bound} \label{app:upper}
First we bound the drop of the potential $\phi$.
\begin{proposition}\label{pro:potentialdrop}
For every $t>0$ we have
\begin{align*}
\E(\phi(\xi(t+1))|\xi(t)) \leq \left(1-\frac{(1-\alpha)(1-\lambda_2)[2\alpha+(1-\alpha)(1+\lambda_2)(1-k^{-1})]}{n}\right)\phi(\xi(t)).
\end{align*}
Moreover, for $\xi(0) = f_2$, the eigenvector associated to the second eigenvalue of the transition matrix $P$, $\lambda_2$, we have
\begin{align*}
\E(\phi(\xi(t)))\geq \left(1-\frac{1-\alpha}{n}(1-\lambda_2)\right)^{2t}
\end{align*}
\end{proposition}
\begin{proof}
For the ease of presentation we write $\xi' = \xi(t+1)$ and $\xi = \xi(t)$ in the proof of this lemma.  Then 
\begin{align}
\|\xi'\|_{\pi}^2 &= \left\| \left(I+ \frac{1-\alpha}{k}\sum_{i=1}^k (\delta_{XY_i}-\delta_{XX})\right)\xi\right\|_{\pi}^2\nonumber\\
&=\|\xi\|_{\pi}^2 +\frac{2(1-\alpha)}{k} \left \langle \xi,\left (\sum_{i=1}^k (\delta_{XY_i}-\delta_{XX})\right)\xi  \right \rangle_{\pi}+ \left( \frac{1-\alpha}{k}\right)^2 \left\| \left( \sum_{i=1}^k (\delta_{XY_i}-\delta_{XX})\right)\xi\right\|_{\pi}^2\label{eqn:pot1}
\end{align}

We now take the conditional expectation (which for shortness we just denote by $\E$ instead $\E(\cdot|\xi(t))$). Thus, we essentially just need to take expectation on the random nodes. The only terms involving randomness are the second and third, so we will focus on them. For the second term, linearity of expectation, \cref{lemma:basicExpectationVec2}.(\ref{item:bev21}), and \cref{lemma:basicExpectationVec2}.(\ref{item:bev22}) yield

\begin{align}\label{eqn:pot2}
\E\left(\left\langle \xi,\left (\sum_{i=1}^k (\delta_{XY_i}-\delta_{XX})\right)\xi  \right \rangle_{\pi}\right) &=\frac{k}{n}\left\langle \xi,(P-I)\xi  \right \rangle_{\pi} = -\frac{k}{n}\left\langle \xi,(I-P)\xi  \right \rangle_{\pi}
\end{align} 

\begin{align}
\left\| \left( \sum_{i=1}^k (\delta_{XY_i}-\delta_{XX})\right)\xi\right\|_{\pi}^2&= \left\| \sum_{i=1}^k \delta_{XX}\xi\right\|_{\pi}^2-2\left \langle \sum_{i=1}^k \delta_{XY_i}\xi,\sum_{i=1}^k \delta_{XX}\xi  \right \rangle_{\pi}+\left\| \sum_{i=1}^k \delta_{XY_i}\xi\right\|_{\pi}^2.\label{eqn:pot3}
\end{align}
Then, \cref{lemma:basicExpectationVec2}.(\ref{item:bev21}) yields
\begin{align}\label{eqn:pot31}
\E \left\| \sum_{i=1}^k \delta_{XX}\xi\right\|_{\pi}^2&= \frac{k^2}{n}\|\xi\|_{\pi}^2,
\end{align}
\cref{lemma:basicExpectationVec2}.(\ref{item:bev22})
\begin{align}
\E \left(\left \langle \sum_{i=1}^k \delta_{XY_i}\xi,\sum_{i=1}^k \delta_{XX}\xi  \right \rangle_{\pi} \right)&= \E\left( \left \langle \sum_{i=1}^k \delta_{XY_i}\xi,\sum_{i=1}^k \xi  \right \rangle_{\pi}\right)
= \frac{k^2}{n} \left \langle P\xi,\xi  \right \rangle_{\pi}\label{eqn:pot32}
\end{align}
and, finally, \cref{lemma:basicExpectationVec2}.(\ref{item:bev23}) and \cref{lemma:basicExpectationVec2}.(\ref{item:bev24}) yield
\begin{align}
\E\left(\left\| \sum_{i=1}^k \delta_{XY_i}\xi\right\|_{\pi}^2\right)& = \E\left(\sum_{i=1}^k \left\|\delta_{XY_i}\xi \right\|_{\pi}^2\right)+ \E \left(\sum_{i=1}^n \sum_{j=1}^n\ind{i\neq j} \left\langle \delta_{XY_i}\xi, \delta_{XY_j}\xi \right \rangle_{\pi}\right)
=\frac{k}{n}\|\xi\|_{\pi}^2 + \frac{k^2-k}{n}\left \langle \xi, P^2\xi \right \rangle_{\pi}.\label{eqn:pot33}
\end{align}
By replacing equations~\eqref{eqn:pot31}, \eqref{eqn:pot32}, and \eqref{eqn:pot33} into \eqref{eqn:pot3} we get
\begin{align}\label{eqn:pot3p}
\E\left(\left\| \left( \sum_{i=1}^k (\delta_{XY_i}-\delta_{XX})\right)\xi\right\|_{\pi}^2\right)&=\frac{k^2}{n}\|\xi\|_{\pi}^2 -2\frac{k^2}{n}\left \langle P\xi, \xi \right \rangle_{\pi}+ \frac{k}{n}\|\xi\|_{\pi}^2 + \frac{k^2-k}{n}\left \langle \xi, P^2\xi \right \rangle_{\pi}\\
&= \frac{2k^2}{n}\left \langle \xi, (I-P)\xi \right \rangle_{\pi} -\frac{(k^2-k)}{n}\left \langle \xi, (I-P^2)\xi \right \rangle_{\pi}
\end{align}
Finally, we insert equations~\eqref{eqn:pot2} and \eqref{eqn:pot3p} in equation~\eqref{eqn:pot1} and we get
\begin{align}
\E\left(\|\xi'\|_{\pi}^2\right) &= \|\xi\|_{\pi}^2 - \frac{2(1-\alpha)}{n}\left \langle \xi,(I-P)\xi \right \rangle_{\pi}+\frac{(1-\alpha)^2}{n}\left(2\left \langle \xi, (I-P)\xi \right \rangle_{\pi}- \left(1-\frac{1}{k}\right)\left \langle \xi, (I-P^2)\xi \right \rangle_{\pi} \right)\nonumber\\
&=\|\xi\|_{\pi}^2 - \frac{2(1-\alpha)\alpha}{n}\left \langle \xi,(I-P)\xi \right \rangle_{\pi}-\frac{(1-\alpha)^2}{n}\left(1-\frac{1}{k}\right)\left\langle \xi, (I-P^2)\xi \right \rangle_{\pi}\label{eqn:FinalEPot}
\end{align}

Recall that  $P\xi = \sum_{i=1}^n \lambda_i \left \langle \xi, f_i \right \rangle_{\pi}f_i$, where
$f_i$ are orthonormal under $\left \langle \cdot, \cdot \right \rangle_{\pi}$, and $\lambda_i$ are the eigenvalues $1\geq \lambda_1>\lambda_2\geq \ldots\geq \lambda_n>0$ (See Lemma 12.2 in \cite{levin2017markov}). Hence,
\begin{align*}
\left \langle \xi,(I-P)\xi \right \rangle_{\pi} = \left \langle \sum_{i=1}\left \langle \xi, f_i \right \rangle_{\pi}f_i, \sum_{i=1}^n (1-\lambda_i) \left \langle \xi, f_i \right \rangle_{\pi}f_i \right \rangle_{\pi} \geq \sum_{i=2}^n (1-\lambda_i) \left \langle \xi,f_i \right \rangle_{\pi}^2 \geq (1-\lambda_2) \sum_{i=2}^n  \left \langle \xi,f_i \right \rangle_{\pi}^2
\end{align*}
and using that $f_1 = \one$, we have
\begin{align*}
\left \langle \xi,(I-P)\xi \right \rangle_{\pi} \geq (1-\lambda_2) \sum_{i=2}^n  \left \langle \xi,f_i \right \rangle_{\pi}^2 = (1-\lambda_2 )\left(\|\xi\|_{\pi}^2-\left \langle \one, \xi \right \rangle_{\pi}^2 \right) = (1-\lambda_2)\phi(\xi)
\end{align*}
and similarly,
\begin{align*}
\left \langle \xi,(I-P^2)\xi \right \rangle_{\pi} \geq  \left (1-\lambda_2^2 \right)\left(\|\xi\|_{\pi}^2-\left \langle \one, \xi \right \rangle_{\pi}^2 \right) = (1-\lambda_2^2)\phi(\xi).
\end{align*}
Therefore, by subtracting  $\E\left \langle \one, \xi' \right \rangle_{\pi}^2 \geq  \left \langle \one, \xi \right \rangle_{\pi}^2$ to, respectively, the l.h.s. and the r.h.s. of \cref{eqn:FinalEPot}, and by applying the previous inequalities, we have, by the definition of $\phi$ (see \cref{eqn:defiPot})
\begin{align*}
\E(\phi(\xi(t+1)))&\leq \phi(\xi(t))  -\frac{2(1-\alpha)\alpha}{n}\left \langle \xi,(I-P)\xi \right \rangle_{\pi}-\frac{(1-\alpha)^2}{n}\left(1-\frac{1}{k}\right)\left\langle \xi, (I-P^2)\xi \right \rangle_{\pi}\\
&\leq \phi(\xi(t)) -\frac{2\alpha(1-\alpha)(1-\lambda_2)}{n}\phi(\xi(t))- \frac{(1-\alpha)^2\left(1-\lambda_2^2\right)}{n}\left(1-\frac1k\right) \phi(\xi(t))\\
&= \left(1-\frac{2\alpha(1-\alpha)(1-\lambda_2)+(1-\alpha)^2\left(1-\lambda_2^2\right)(1-1/k)}{n}\right)\phi(\xi(t))\\
&=\left(1-\frac{(1-\alpha)(1-\lambda_2)[2\alpha+(1-\alpha)(1+\lambda_2)(1-k^{-1})]}{n}\right)\phi(\xi(t)),
\end{align*}
\end{proof}

\subsection{Lower Bound}\label{app:lower}
\begin{proposition}
Let $k\geq 1$.
For the $\mone$, set $\xi(0)=n \cdot f_2(P)$, where $f_2(P)$ is the second eigenvector of the random walk transition matrix, then we have 
\begin{align*}
\E(T_\epsilon)=\Omega\left(\frac{n \log( n\twonorm{\xi(0)}/\epsilon)}{(1-\alpha)(1-\lambda_2(P))}\right).
\end{align*}

For the $\mtwo$, set $\xi(0)=n\cdot f_2(L)$, where $f_2(L)$ is the second eigenvector of the Laplacian, then we have 
\begin{align*}
\E(T_\epsilon)=\Omega\left(\frac{m \log(n\twonorm{\xi(0)}/\epsilon)}{(1-\alpha)(\lambda_2(L))}\right).
\end{align*}
\end{proposition}
\begin{proof}
Set $\beta = n$.
We prove the proposition for the $\mone$. The proof for the $\mtwo$ is analogous. 


Using the notation of \cref{eq:hut}, we have
\begin{align}
\xi(t+1) = \left(I+ \frac{(1-\alpha)}{k}\sum_{i=1}^k (\delta_{XY_i}-\delta_{XX})\right)\xi(t)
\end{align}
By linearity of expectation,

\begin{align}
\E(\E(\xi(t+1)\ |\ \xi(t))) &= \left(I+ \frac{(1-\alpha)}{n}(P-I)\right)\cdot\E(\xi(t))
\end{align}
Hence,
\begin{align}
\E(\xi(t+1)) 
&= \left(I- \frac{(1-\alpha)}{n} (I-P)\right)\cdot\E(\xi(t)) = \left(I- \frac{(1-\alpha)}{n}(I-P)\right)^{t+1}\cdot
\xi(0)
\end{align}
Let $q_2 = \left(1-\frac{1-\alpha}{n}(1-\lambda_2(P)\right)$. 
For our choice of $\xi(0)$, using that $Pf_2 = \lambda_2f_2$, we get

\begin{align}
\E(\xi(t)) = \left(I- \frac{1-\alpha}{n}(I-P)\right)^{t}\beta\cdot f_2 = \left(1-\frac{1-\alpha}{n}(1-\lambda_2(P))\right)^t\beta\cdot f_2 = q_2^t\cdot  f_2  .
\end{align}

Consider the potential $\phi$ and recall that \
\begin{align*}
\phi(\xi(t)) = \frac{1}{2}\sum_{x,y\in V}\pi_x\pi_y(\xi_x(t)-\xi_y(t))^2
\end{align*}
showing that $\phi$ is a convex potential.  Thus, we can use linearity of expectation and Jensen's inequality to conclude that 
\begin{align}
\E(\phi(\xi(t))) \geq \frac{1}{2}\sum_{x,y\in V}\pi_x\pi_y(\E(\xi_x(t)-\xi_y(t)))^2 
\end{align}
Finally, using the definition of the potential we have by \cref{eqn:defiPot} 
\begin{align}
\E(\phi(\xi(t))) \geq \beta^2 q_2^{2t}\left(\|f_2\|_{\pi}^2-\left \langle \one, f_2 \right \rangle_{\pi}^2\right) =  \beta^2 q_2^{2t}
\end{align}
The last equality holds for the orthonormality of eigenvectors of matrix $P$ previously mentioned.
Note that $\twonorm{\xi(0)}=\beta^2$.
After $t=\frac{n \log(n\twonorm{\xi(0)})}{8(1-\alpha)(1-\lambda_2(P))}$ we have

\begin{align} \E(\phi(\xi(t))) &\geq  q_2^{2t} \beta^2
\geq \beta^2
\left(1-\frac{1-\alpha}{n}(1-\lambda_2)\right)^{ \frac{n}{(1-\alpha)(1-\lambda_2(P))} \frac{\log(n\twonorm{\xi(0)}/\epsilon)}{8}  } 
\geq \beta^2 e^{-\frac{\log(n\twonorm{\xi(0)}/\epsilon)}{4}}(1-o(1)) \\ 
&\geq  \beta^2 \left(\frac{\epsilon}{n\twonorm{\xi(0)}}\right)^{\frac{1}{4}}(1-o(1))
> \epsilon,
\end{align}

\end{proof}

\subsection{Proof of Theorem~\ref{thm:mone}.(1)} \label{app:thmMone}
By \cref{pro:potentialdrop},

\begin{align}
\E(\phi(\xi(t+1))) &\leq \left(1-\frac{(1-\alpha)(1-\lambda_2)[2\alpha+(1-\alpha)(1+\lambda_2)(1-k^{-1})]}{n}\right)\phi(\xi(t))\\
&=\left(1-\frac{(1-h\lambda_2)}{n}\right)^t \phi(\xi(t)),
\end{align}
for some suitable constant $h$,
since $\lambda_2 \geq -1$ and $\alpha$ constant with $\alpha \in (0,1)$.
\begin{align*}
\E(\phi(\xi(t))) \leq  \left(1-\frac{h(1-\lambda_2)}{n}\right)^t \phi(\xi(0)).
\end{align*}

By \cref{eq:ponies}, we have that $\phi(\xi(0))\leq n \twonorm{\xi(0)}.$ 
Thus setting $t=3 \frac{n\log(n\twonorm{\xi(0)}/\epsilon)}{h(1-\lambda_2)} =\frac{n\log((n\twonorm{\xi(0)}/\epsilon)^3)}{h(1-\lambda_2)} $, we get

\begin{align}
\E(\phi(\xi(t))) \leq \left( \frac{\epsilon}{n\twonorm{\xi(0)}}\right)^3 \phi(\xi(0))\leq \frac{ \epsilon}{n^2}.
\end{align}
By Markov inequality,
$\Pr(\E(\phi(\xi(t))) \geq \epsilon) \leq 1/n^2$.

The proof of the lower bound can be found in \cref{app:lower}.

\section{\mone -- Concentration Bounds} \label{app:nodeConc}
\omitthis{
\subsection{Proof of \texorpdfstring{\cref{thm:limitingVarExact1}}{}} \label{app:proof4-1}
\begin{proof}
The random variable $\avg(t) = \frac{1}{n}\sum_{x\in V} \xi_x(t)$ is a Martingale (\cref{lem:NodeMartingale}).
Due to our convergence results, we know that as $t \rightarrow \infty$, all node values 
$\xi_x(t)$ converge to the same value. 
Consequently, $\avg(t)$ converges 
to the same value, which we denote by 
$\avg(\infty)$.
This is a random variable 
with expectation $\E[\avg(\infty)] = \avg(0)$, 
since by the Martingale property, for each $t\ge 0$, 
$\E[\avg(t)] = \avg(0)$.
We want to show that the actual value of $\avg(\infty)$ is likely to be close to $\avg(0)$. We do this by showing that $\var(\avg(\infty)) = \lim_{t \rightarrow \infty} \var(\avg(t))$ is small.
We start by 
recalling the assumption that 
$\avg(0)=0$ and using linearity of expectation to obtain
\begin{align*}
\var(\avg(t)) &= \E[\avg(t)^2] - (\E[\avg(t])^2
= \E[\avg(t)^2] = \frac{1}{n^2}
\sum_{x,y\in V} \E[\xi_x(t)\xi_y(t)].
\end{align*}

Fixing $t=T$ arbitrarily, 
Propositions~\ref{prop:duality} 
and~\ref{prop:variance_both_processes1} imply that for any pair of nodes $x$ and $y$,
the three products
$\xi_x(T) \xi_y(T)$, $W^{(x)}(T) W^{(y)}(T)$ and $\widetilde{W}^{(x)}(T) \widetilde{W}^{(y)}(T)$ have the same expectation, so 
\begin{align*}
    \var(\avg(T)) &= \frac{1}{n^2}
    \sum_{x,y} \E[\xi_x(T)\xi_y(T)]\\
    &=\frac{1}{n^2}\sum_{x,y} \E[W^{(x)}(T)W^{(y)}(T)]
    =\frac{1}{n^2}\sum_{x,y} \E[\widetilde{W}^{(x)}(T)\widetilde{W}^{(y)}(T)]. 
\end{align*}
%
%
Let $\xi = \xi(0)$, then by \cref{lem:crwTOstationarydistr}
and \cref{prop:stationaryDistrib} we get
\begin{align*}
    \var(\avg(T)) &= 
    \frac{1}{n^2}\sum_{x,y} \left(\sum_{u,v}\mu(u,v) \xi_u\xi_v \pm 1/n^5\right) 
    =  \sum_{u,v}\mu(u,v) \xi_u\xi_v \pm 1/n^5\\
%
%
%
%
%
    &= \mu_0 \sum_{(u,u) \in S_0} \xi_u^2\ + \mu_1 \sum_{(u,v) \in S_1} \xi_u\xi_v + \mu_+ \sum_{(u,v) \in S_+} \xi_u\xi_v \pm 1/n^5 \\
    & =  (\mu_0-\mu_+) \sum_{(u,u) \in S_0} \xi_u^2 + (\mu_1-\mu_+) \sum_{(u,v) \in S_1} \xi_u\xi_v + \mu_+ \sum_{u,v \in V} \xi_u\xi_v \pm 1/n^5\\
    &= (\mu_0-\mu_+) \sum_{u \in V} \xi_u^2 + (\mu_1-\mu_+) \sum_{(u,v) \in E} \xi_u^2 \pm 1/n^5, 
\end{align*}
using for the last equation that
$\sum_{u,v}\xi_u\xi_v=
(\sum_{u}\xi_u)(\sum_u\xi_u)= 0$
(the assumption that $\avg(0)=0$).
Thus \cref{eq:concentrationFirst2} holds.
For \cref{eq:concentrationFirst2}, observe 
that $\var(\avg(t))$ is non-decreasing.


\end{proof}
}

\omitthis{
\subsection{Proof of \texorpdfstring{\Cref{lem:dualitySupportLemma}}{}}\label{app:proof4-2}
\begin{proof}
Let $F(t)$, $t=1,2,\ldots,T$,
denote the update matrix in step $t$ of the averaging process corresponding to the edge choice $\newChi(t)$ (`$F$' indicates that we are 
referring 
to the \emph{forward process}).
That is, $\xi(t) = F(t)\, \xi(t-1)$, 
viewing vectors $\xi(t)$'s as column vectors.
Thus if $\newChi(t) = (S(t),x(t))$, then matrix $F(t)$ is such that
row $x$ is the vector with $\alpha$ at entry $x$ (diagonal), with $k$ elements $\{\alpha_1, \ \alpha_2, \dots, \alpha_k\}, \ \sum_{i=1}^k\alpha_i = 1 - \alpha$, at entries $u \neq x$, respectively, and~$0$ elsewhere. 
All other rows have $1$ at the diagonal 
and $0$ elsewhere.
That way 
$\xi_u(t) = \alpha \xi_u(t-1) + (1-\alpha)\xi_v(t-1)$,
and $\xi_w(t) = \xi_w(t-1)$ for each $w\neq u$,
as required.
Observe that, 
\begin{eqnarray}
\xi(\newChi) & = & \xi(T) 
\; = \; F_{T} (F_{T-1} ( \ldots (F_{2}
(F_1 \, \xi(0))\ldots))
= \left(\prod_{t=1}^{T} F(T-t)\right) \xi(0).
\label{eq:avg1}
\end{eqnarray}

The expectation over one step of the process, for the forward matrix $F$, is a matrix $Q_1= \E(F)$, with entries $Q_1(x,y)$ given by
\begin{align} 
    Q_1(x,y) &= \; \pi_x (1-\alpha)P(x,y)
    \label{hbfqei}\\
    Q_1(x,x) &= \; (1-\pi_x)+\a \pi_x+(1-\a)\pi_x P(x,x).\label{hchwec}
\end{align}
$Q_1$ hence represents the transition matrix of a random walk process $X(s)$ where the transition from node $x$ to node $y$ is
\begin{equation*}\label{eq:q1}
    \Prob(X(s)=y \mid X(s)=x)=Q_1(x,y).
\end{equation*}
In \cref{sec:relatingProcesses} we will show that, from this, it follows the construction of another transition matrix $Q_2 = Q$ for the combination of two random walks on $G$.
Let $B_t$, $t = 1, 2, \ldots, T$ 
(`$B$' indicates the \emph{backward process})
be the update matrix 
in step $t$ of the random walk process corresponding
to the edge choice $\newChi_{T-t+1} = (u',v')$,
that is, for any fixed node $x$,
$r^{(x)}(t) = B_t\, (r^{(x)}(t-1))$,
viewing vectors $r^{(x)}(t)$'s as column vectors.
This means that matrix $B_t$ is such that 
column $u'$ is the vector with $\alpha$ at entry $u'$ (diagonal), with $k$ elements $\{\alpha_1, \ \alpha_2, \dots, \alpha_k\}, \ \sum_{i=1}^k\alpha_i = 1 - \alpha$, at entries $v' \neq u'$, and $0$ elsewhere. 
All other columns have $1$ at the diagonal 
and $0$ elsewhere.
That way $r^{(x)}_{u'}(t) = \alpha r^{(x)}_{u'}(t-1)$,
$r^{(x)}_{v'}(t) = r^{(x)}_{v'}(t-1) 
+ (1-\alpha) r^{(x)}_{u'}(t-1)$,
and $r^{(x)}_w(t) = r^{(x)}_w(t-1)$ for each 
$w\not\in\{u',v'\}$.
Since nodes $u'$ and $v'$ in step $t$ 
of the random walk process are the same as 
nodes $u$ and $v$ in step $T-t+1$ of
the averaging process, then
$B_t  = (F_{T-t+1})^\top$, where the superscript $\top$ denotes 
matrix transposition.
We have
\begin{eqnarray}
R(T) & = & 
\left(\prod_{t=1}^{T} B(T-t)\right) R(0),
\end{eqnarray}
where $r^{(x)}(0)$ is the unit vector with $1$ 
at position $x$ (indicating the starting node $x$
of the random walk).
Hence we have
\begin{eqnarray*}
    W^{(\newChi^R)} &= W(T) = \xi(0)^{\top}R(T) = \xi(0)^\top\left(\prod\limits_{t=1}^{T} B(T-t)\right)R(0)\\
    &= \xi^\top(0)\left(\prod\limits_{t=1}^{T} B(t)\right) = \xi^\top(0)\left(\prod\limits_{t=1}^{T} F(T-t)\right)^\top \\
    &= \left(\left(\prod\limits_{t=1}^{T} F(T-t)\right)\xi(0) \right)^\top= \xi^\top(T)
\end{eqnarray*}
\end{proof}
}

\subsection{Proof of \texorpdfstring{\cref{prop:stationaryDistrib}}{}}\label{sec:proofOfmu}

\begin{proof}

The $Q$ chain has $|V^Q| = n^2$ many states and any state belongs to exactly one of the three types of sets $S_0, S_1, S_+$.
The stationary distribution $\mu$  is obtained from a system of $n^2$ equations in $n^2$ unknowns, each of the unknowns belonging to one of the three types of sets.
Since the $Q$-chain is irreducible and aperiodic, by definition $\mu$  satisfies the equation $\mu = \mu Q$ and
\begin{align} \label{eq:muCompositionGeneral}
    \mu(u, v) &= \sum_{(x, y) \in \mathcal{N}(u, v)} \mu(x, y)Q((x, y), (u, v)).
\end{align}
Note that for $k>1$ the $Q$-chain is not reversible since the state space of the $Q$-chain is not
bi-directional. Hence, the detailed balanced equations of the form $\mu_iQ_{ij} = \mu_jQ_{ji}$ do not hold. 
In particular, while a transition $(x, y)\rightarrow(u, v)$ with $(x, y)\in S_0$ and $(u, v)\in S_+$ is possible (for example going from distance $0$ to distance $2$), the transition in the opposite direction is not possible ($(u, v)\cancel{\rightarrow}(x, y)$).

In the following we fix a state $(u, v)$ and describe how it can be reached from other states in the $Q$ chain. Let $c=c(u, v)$ be the number of common neighbours of $u$ and $v$. Note that $c=0$ if  $u$ and $v$ have a distance $i > 2$.
For simplicity of notation, we define $f(x, y) := \mu(x, y)Q((x, y),(u, v))$ and $\mathcal{N} := \mathcal{N}(u, v)$, where  $\mathcal{N}(u, v)$ is the set of neighbours of $(u, v)$. \cref{eq:muCompositionGeneral} can be split into different parts, each accounting for a different set $S_i$ of states
\begin{align} \label{eq:muCompositionDetailed}
          \mu(u, v) &= \sum_{\substack{(x,y) \in \\ \mathcal{N} \cap S_0}} f(x,y) + \sum_{\substack{(x,y) \in \\ \mathcal{N} \cap S_1}} f(x,y) + \sum_{\substack{(x,y) \in \\ \mathcal{N} \cap S_+}} f(x,y)
\end{align}
We define $\mu_0 := \mu(u,v)$ if $(u,v) \in S_0$, $\mu_1 := \mu(u,v)$ if $(u,v) \in S_1$ or $\mu_+ := \mu(u,v)$ for $(u,v) \in S_+$.

\paragraph{Case 1: $(u,v) \in S_0$.} In this case  \cref{eq:muCompositionDetailed} becomes \cref{eq:mu0}. The first element accounts for the transitions as described in \cref{eq:matrixQ2_2} and \cref{eq:matrixQ2_5}, while the second element accounts for the transitions as described in \cref{eq:matrixQ3_1}. 
Note that there is no transition from $\mu_+$, since, in the case where the random walks are on different nodes, only at most one of them can move and cannot get to distance zero.
\begin{align}
        \mu_0 &= \sum_{(x, y) \in \mathcal{N} \cap S_0}
        f(x, y) \; + \sum_{(x, y) \in \mathcal{N} \cap S_1}
        f(x, y) \notag \\
        &= \mu_0\left(\frac{n-1}{n} + \pi \alpha^2 + d\pi (1-\alpha)^2\cdot \frac{1}{kd}\right) + \mu_1\left(2 \pi (1-\alpha)\cdot \frac{1}{d} \cdot d\right) \label{eq:mu0}
\end{align}

\paragraph{Case 2: $(u,v) \in S_1$.}
In this case, we can get to $(u, v)$ from states in $S_0$, $S_1$ or $S_+$.

First we consider states in $S_0$.
This happens when $x$ is a common neighbour of $u$ and $v$ and there is a transition $(x, x) \rightarrow (u, v)$ with $(x, x)\in S_0$. Recall that $c$ is the number of common neighbours of $u$ and $v$. If there are $c$ common neighbours, there could be $c$ many transitions of this type (see \cref{fig:from0}).

Next, we consider the case where we are coming from $S_1$. We can get from a state $(u, x)\in S_1$ ($(x, v)\in S_1$) to $(u, v)$ if and only if $(x, v)\in S_1$ ($(u, x)\in S_1$) (see \cref{fig:from1}). 
Again, this happens when $x$ is a common neighbour of $u$ and $v$.

Finally, we consider states in $S_+$. If $x$ is not a common neighbour of $u$ and $v$, then there is a transition $(u, x)\rightarrow (u, v)$ and $ (u, x) \in S_+$.
Therefore $u$ has $c$ neighbours $x$ for which $(u, x) \rightarrow (u, v)$ and $(u, x) \in S_1$, one direct link to $v$, and $d - c - 1$ neighbours $x$ with a transition $(u, x)\rightarrow (u, v)$ and $(u, x) \in S_+$. The same observation applies to node $v$. 

Putting the 3 cases together \cref{eq:muCompositionDetailed} becomes
\begin{align}
    \begin{split}
        \mu_1=& \sum_{(x, y) \in \mathcal{N} \cap S_0}\mu(x, y)Q((x, y), (u, v)) \; + \sum_{(x, y) \in \mathcal{N} \cap S_1}\mu(x, y)Q((x, y), (u, v))\\
        &+ \sum_{(x, y) \in \mathcal{N} \cap S_+}\mu(x, y)Q((x, y), (u, v)\notag
    \end{split} \\
    \begin{split}
        =& \; \mu_0\left(2\pi\alpha(1-\alpha)\frac{1}{d} + \pi c (1-\alpha)^2 \cdot\frac{k-1}{kd(d-1)}\right) + \mu_1 \left(\frac{n-2}{n} + 2\pi \alpha + 2\pi c(1-\alpha)\cdot \frac{1}{d}\right)\\
        &+ \; \mu_+\left(2\pi(d-c-1)(1-\alpha)\cdot\frac{1}{d}\right). \label{eq:mu1}
    \end{split}
\end{align}

The elements of \cref{eq:mu1} account for the transitions as described in \cref{eq:matrixQ2_1}, \cref{eq:matrixQ2_3}, \cref{eq:matrixQ2_4} and \cref{eq:matrixQ3_3}.

In order to check that $c=c(u, v)$ vanishes simultaneously in the evaluation of \cref{eq:mu1} we can extract the coefficients of $c$ on the LHS and RHS to verify this. After some cancellation and by substituting $\mu_*$ taken from \cref{thm:limitingVarExact1}, we obtain 
\begin{align*}
   0 &= \frac{(1-\a)(k-1)}{k(d-1)}\mu_0 + 2\mu_1 -2\mu_+\\
   &= (1-\alpha)(k-1) + (d-1)\gamma - (d\gamma - 2\alpha k),
\end{align*}
Which can be checked to be correct.

\paragraph{Case 3: $(u, v) \in S_+$.} In this case we obtain
\begin{align}
    \mu_+ =& \sum_{(x, y) \in \mathcal{N} \cap S_0}
    f(x, y)
    + \sum_{(x, y) \in \mathcal{N} \cap S_1}
    f(x, y)
    + \sum_{(x, y) \in \mathcal{N} \cap S_+}
    f(x, y)\notag \\
    =& \mu_+ \left(\frac{n-2}{n} + 2\pi \alpha + \pi 2(d-c)(1-\alpha)\frac{1}{d}\right) + \mu_1 \left(2c\pi (1-\alpha)\frac{1}{d}\right) +  \mu_0 \left(c \pi (1-\alpha)^2\frac{k - 1}{kd(d-1)}\right) \label{eq:mu2}
\end{align}

The elements of \cref{eq:mu2} are again described in \cref{eq:matrixQ2_1}, \cref{eq:matrixQ2_3}, \cref{eq:matrixQ2_4} and \cref{eq:matrixQ3_3}. It can be checked that $c$ cancels out also in this case, following the previous argument.

For $dis(u,v) > 2$, then $c=0$. \cref{eq:mu2} still holds and becomes
\begin{align*}
    \mu_+ =& \mu_+\left(\frac{n-2}{n} + 2\pi \alpha + \pi 2(1-\alpha)\right) \\
    1 =& 1 - \pi 2 + 2\pi \alpha + \pi 2 - 2\pi \alpha
\end{align*}

The final equation we need is the following, which tells us that the stationary distribution needs to sum to $1$. As $|S_0|=n$, and $|S_1|=2|E|$, for d-regular graphs this is equivalent to
\begin{equation}\label{eq:sumto1}
        1 = n\mu_0 + 2 |E| \mu_1 + (n^2 - 2|E| - n)\mu_+ = n\mu_0 + nd \mu_1 + n(n - d - 1)\mu_+ .
\end{equation}

If we now input the values for $\mu_+$ from \cref{eq:threeSets} in \cref{eq:mu0},  \cref{eq:mu1} and \cref{eq:mu2} the equations will be satisfied and $c$ will disappear, proving their validity for any of the $n^2$ equations of the system.
\begin{figure}
\begin{minipage}{.45\textwidth}
    \centering
    \includegraphics[width=.9\textwidth]{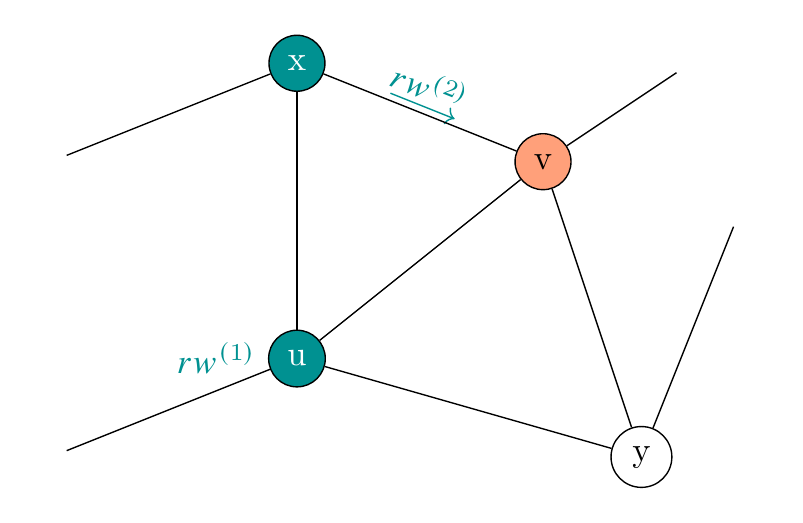} 
    \caption{The random walks are located in the blue nodes. $rw^{(1)}$ is in $u$, the state $(u, x) \in S_1$. $rw^{(2)}$ moves from $x$ to $v$, obtaining the state $(u, v) \in S_1$}
    \label{fig:from1}
\end{minipage}\qquad
\begin{minipage}{.45\textwidth}
    \centering
    \includegraphics[width=.9\textwidth]{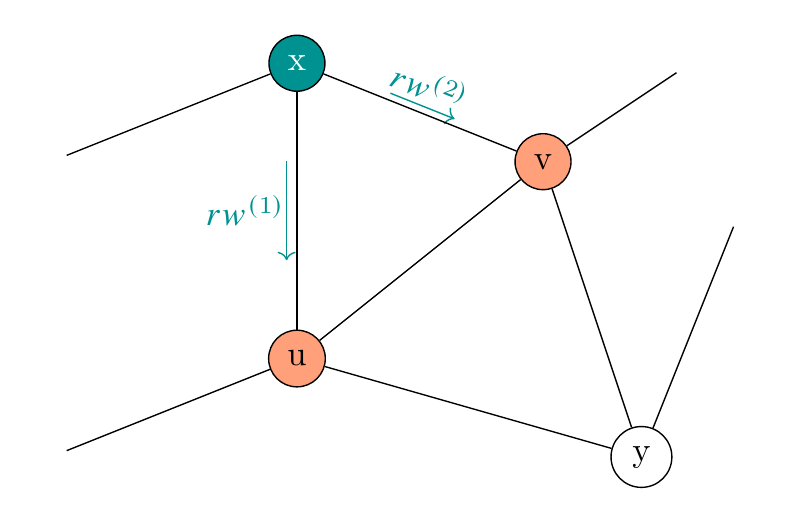} 
    \caption{Both $rw^{(1)}$ and $rw^{(2)}$ are in $x$, the state $(x, x) \in S_0$. Later, $rw^{(1)}$ moves from $x$ to $u$ and $rw^{(2)}$ moves from $x$ to $v$, obtaining again $(u, v) \in S_1$}
    \label{fig:from0}
\end{minipage}
\end{figure}
\end{proof}

\section{\mtwo} \label{sec:edgeModelProofs}
\subsection{Proof of \cref{thm:mtwo}(1) -- Convergence Time}
We start with the following proposition.
\begin{proposition}\label{pro:Colin}
Let $G$ be a connected graph with $m$ edges and let $\avg(0)= \frac{1}{n} \sum_i \xi_i(0)$. Let $\bar{\phi}_V(\xi(t)) = \frac{1}{2n}\sum_{x,y \in V} (\xi_x(t)-\xi_y(t))^2$.
\begin{enumerate}[(i)]
\item $\avg(t) $ is a martingale and thus $\E( \avg(t))=\avg(0)$. 
\item \[
\E (\bar{\phi}_V(\xi(t))) \le \brac{1-\frac{\a(1-\a)}{m} \l_2(L)}^t \bar{\phi}_V(\xi(0)).
\]
\end{enumerate}
\end{proposition}
\begin{proof}
First we prove (i). Again, we write $\xi$ instead of $\xi(t)$  and  $\xi'$ instead of $\xi(t+1)$ in  the  proof  of  this lemma. 
Further, denote the value of a variable $Z(t)$ by $Z=Z(t)$ and $Z'=Z(t+1)$. By linearity of expectation
\[
\E(\xi'_x \mid \xi, \text{edge } (x,y) \text{ chosen})=\frac 12 \xi_x+ \frac 12 (\a \xi_x+ (1-\a) \xi_y)
=\frac 12 (1+\a)\xi_x+ \frac12 (1-\a)\xi_y.
\]
Thus, for $x \in V$, 
\begin{flalign*}
\E (\xi_x' \mid \xi)=& \brac{1-\frac{d_i}{m}}\xi_x+ \frac 1m \brac{\sum_{y \sim x} \frac 12 (1+\a)\xi_x+ \frac12 (1-\a)\xi_y}
=\brac{1-(1-\a)\frac{d_x}{2m}}\xi_x+\frac{1-\a}{2m} \sum_{y \sim x} \xi_y.
\end{flalign*}
Let $Y=\sum_{x\in V} \xi_x$ and  $Y'=\sum_{x\in V} \xi'_x$. Then
\begin{flalign*}
\E(Y' \mid \xi)=&\sum_{x\in V} \brac{\brac{1-(1-\a)\frac{d_x}{2m}}\xi_x+\frac{1-\a}{2m} \sum_{y \sim x} \xi_y}
=\sum_{x\in V} \xi_x- \frac{1-\a}{2m} \sum_{x\in V} d_x \xi_x +\frac{1-\a}{2m}\sum_{y\in V} \xi_y \sum_{x \sim y} 1=Y.
\end{flalign*}

Now we are ready to prove (ii). Given $x\in V$, by linearity of expectation we get
\[
\E\left((\xi'_x)^2 \mid \xi, \text{edge } (x,y) \text{ chosen}\right)=\frac 12 \xi_x^2+ \frac 12 (\a \xi_x+ (1-\a) \xi_y)^2.
\]
\[
\E\left( (\xi_x')^2 \mid \xi\right)=\xi_x^2\brac{1-\frac{(1-\a^2)}{2m}d_x}
+\sum_{y \sim x} \frac{\a(1-\a)}{m} \xi_x\xi_y
+\sum_{x\sim y} \frac{(1-\a)^2}{2m} \xi_y^2.
\]
Thus 
\begin{flalign}
\E\left(\sum_i (\xi'_x)^2 \mid \xi \right)=& \sum_{x\in V} \xi_x^2\brac{1-\frac{(1-\a^2)}{2m}d_x}+
\frac{\a(1-\a)}{m}\sum_{x\in V} \sum_{y \sim x}\xi_x\xi_y+
\frac{(1-\a)^2}{2m}\sum_{x\in V}  \xi_x^2 d_x\nonumber\\
=&\sum_{x\in V} \xi_x^2\brac{1-\frac{\a(1-\a)}{m}d_x}+
\frac{\a(1-\a)}{m}\sum_i \sum_{y \sim x}\xi_x\xi_y \nonumber\\
=& \sum \xi_x^2 -\frac{\a(1-\a)}{m}
\brac{\xi^{\top}D\xi-\xi^{\top}A\xi} 
= \sum_{x\in V} \xi_x^2 -\frac{\a(1-\a)}{m}
\brac{\xi^{\top}L\xi}\label{LA}
\end{flalign}
Recall that $L$ represents the graph Laplacian of $G$. $L$ is symmetric positive semi-definite, so it has an orthonormal basis $U=(f_1,...,f_n)$ such that $L=U^{\top}\Lambda U$. In particular $f_1=(1/\sqrt{n}) \ul 1$. Thus, $\xi=\sum_{i=1}^n \pbrac{\xi,f_i} f_i$ and $L\xi=\sum_{i=1}^n \l_i\pbrac{\xi,f_i} f_i$,
where $\pbrac{a,b}$ is the standard inner product $\sum a_ib_i$ with no weights.
Hence, as $\l_1=0$,
\[
L\xi \ge \l_2 \sum_{i=2}^n\pbrac{\xi,f_i} f_i, \qquad \text{and}
\]
\begin{equation} \label{eq:LC}
\xi^{\top}L\xi \ge  \l_2 \sum_{i=2}^n \pbrac{\xi,f_i}^2
=\l_2 \brac{\sum_{x\in V} \xi_x^2-\pbrac{\xi,f_1}^2}.
\end{equation}
As $\E\left(\sum_{x\in V} \xi_x\right)^2 \ge \left(\E \sum_{x\in V}\xi_x\right)^2$, and $Y(t)=\sum_{x\in V} \xi_x$ is a martingale, we have
\begin{equation}\label{LB}
\E \pbrac{\xi',f_1}^2=\frac 1n \E\left(\sum_{x\in V} \xi_x'\right)^2 \ge \frac 1n \left(\E \left(\sum _{x\in V}\xi_x'\right)\right)^2
=\frac 1n \left(\sum_{x\in V}\xi_x\right)^2=\pbrac{\xi,f_1}^2.
\end{equation}

We define $\bar{\phi}_V(\xi(t)) = \frac{1}{2n}\sum_{x,y \in V} (\xi_x(t)-\xi_y(t))^2$.
Note that $\bar{\phi}_V(\xi(t))=\sum \xi_x^2- \frac 1n (\sum \xi_x)^2 = \sum \xi_x^2 - \pbrac{\xi',f_1}^2$. Then we use linearity of expectation once more, yielding
\begin{flalign*}
\E(\bar{\phi}_V(\xi')) &= \E\left(\sum (\xi'_x)^2\right)-\E \pbrac{\xi',f_1}^2
\le \E\left(\sum_{x} (\xi'_x)^2\right)- \pbrac{\xi,f_1}^2\\
&\le \sum_{x\in V} \xi_x^2 -\frac{\a(1-\a)}{m}
\brac{\xi^{\top}L\xi} - \pbrac{\xi,f_1}^2
\le \bar{\phi}_V(\xi) \brac{1-\frac{\a(1-\a)}{m} \l_2}.
\end{flalign*}
Where the subsequent steps come from \eqref{LB}, \eqref{LA} and \eqref{eq:LC}, respectively.
It follows that
\[
\E(\bar{\phi}_V(\xi(t))) \le \brac{1-\frac{\a(1-\a)}{m} \l_2}^t \bar{\phi}_V(\xi(0)).
\]

Relating $\bar{\phi_v}$ with $\phi$ one can use the same arguments to conclude as in the proof of \cref{thm:mone}.(1).
\end{proof}
We can now prove \cref{thm:mtwo}(1).
\begin{proof}
The convergence time follows from
\cref{pro:Colin}.
The concentration bound from
\cref{thm:limitingVarExact1}
The time-dependent concentration bound, for general graphs, follows from \cref{COR2}.(iii).
\end{proof}

\subsection{Proof of \cref{thm:mtwo}(2) -- Concentration Bounds}
\extend{The result for the concentration bound can easily be obtained from our results on the $\mone$ (see the beginning of \cref{sec:concentration} for details).}
The concentration follows from \cref{thm:mone} since for regular graphs $\mone$ and $\mtwo$ are identical.

\section{Auxiliary Results}
   \begin{lemma}\label{lemma:basicExpectationVec2}
Let $\E$ denote expectation over the random nodes $X$, and $Y_i$. Then,
\begin{enumerate}
\item\label{item:bev21} \[\E\left(\|\delta_{XX}\xi\|_{\pi}^2\right) =\E\left \langle \xi, \delta_{XX}\xi \right \rangle_{\pi}=\frac{1}{n}\|\xi\|_{\pi}^2 = \frac{1}{n}\left \langle \xi, \xi \right \rangle_{\pi}\]
\item\label{item:bev22} \[\E\left \langle \xi, \delta_{XY_i}\xi \right \rangle_{\pi} = \frac1{n}\left \langle \xi,P\xi \right \rangle_{\pi}\]
\item\label{item:bev23} \[\E\|\delta_{X Y_i}\xi \|_{\pi}^2 = \frac{1}{n} \|\xi\|_{\pi}^2\]
\item\label{item:bev24} for $i\neq j$, 
\[\E\left\langle \delta_{XY_i}\xi, \delta_{XY_j}\xi \right \rangle_{\pi}= \frac{1}{n}\left \langle \xi, P^2\xi \right \rangle_{\pi}\]
\end{enumerate}
\end{lemma}
\begin{proof}
We will go over the items one by one.

\begin{enumerate}
\item
Clearly,
\begin{align*}
\|\delta_{XX}\xi\|_{\pi}^2 =\left \langle \xi, \delta_{XX}\xi \right \rangle_{\pi} = \pi_X\xi_X^2.
\end{align*}
Taking expectation, we have \[\E( \pi_X\xi_X^2) = \frac{1}{n}\sum_{x\in V} \pi_x\xi_x^2 =\frac{1}{n}\left \langle \xi, \xi \right \rangle_{\pi} = \frac{1}{n}\|\xi\|_{\pi}^2 \]

\item We have
\[\E\left \langle \xi, \delta_{XY_i}\xi \right \rangle_{\pi} = \E(\pi_X\xi_X\xi_{Y_i}) =  \frac{1}{n}\sum_{x,y\in V} \pi_x\xi_x\xi_yP(x,y) = \frac{1}{n}\left \langle \xi,P\xi \right \rangle_{\pi}\]


\item We have
\begin{align*}
\E\|\delta_{X Y_i}\xi \|_{\pi}^2= \E(\pi_X
\xi_{Y_i}^2) = \frac{1}{n}\sum_{x,y\in V} \pi_xP(x,y)\xi_y^2 =  \frac{1}{n}\sum_{x,y\in V} \pi_yP(y,x)\xi_y^2 = \frac{1}{n} \sum_{x\in V} \pi_y\xi_y^2
\end{align*}
\item We have,
\begin{align*}
\E\left\langle \delta_{XY_i}\xi, \delta_{XY_j}\xi \right \rangle_{\pi} = \E(\pi_x\xi_{Y_i}\xi_{Y_j}) &= \sum_{x,y,y'\in V} \frac{P(x,y)P(x,y')}{n}\pi_x\xi_y\xi_{y'}\\
&= \sum_{x,y,y'\in V} \pi_y\frac{P(y,x)P(x,y')}{n}\xi_{y}\xi_{y'}\\
&=\sum_{y,y'\in V} \pi_y\frac{P^2(y,y')}{n}\xi_y\xi_{y'}=\frac{1}{n}\left \langle \xi, P^2\xi \right \rangle_{\pi}\\
\end{align*}
\end{enumerate}
This completes the proof.
\end{proof}
Recall that $\D=\D(G)$ is maximum degree of $G$, $i(G)$ is the isoperimetric number defined as
$i(G)= \min_{0 <|S|\le n/2}
|E(S: \ol S)|/{|S|},$  where $S \subset V$ and
$E(S: \ol S)$ are the edges from $S$ to $\ol S=V \sm S$. 

We have the following corollary.

\begin{corollary}\label{COR2}
We have
\begin{enumerate}[(i)]
\item  $\l_2(L) \ge i(G)^2/2\D$ and thus
\[
\E (\phi_V(t)) \le \brac{1-\frac{\a(1-\a)}{2m} \frac{i(G)^2}{\D}}^t \phi(0).
\]
\item Consider the $\mone$. Let $K=\max_i X_i(0)- \min_i X_i(0)$, then  $$\V (M(t)) \le t\left(\frac{\Delta}{2m}  K\right)^2 .$$
\item Consider the $\mtwo$. Let $K=\max_i X_i(0)- \min_i X_i(0)$, then  $$\V (\avg(t)) \le \frac{t}{n^2} K^2 .$$
\end{enumerate}
\end{corollary}
\begin{proof}
First we show (i).
 The value of $\l_2(L)$ is bounded by\footnote{B. Mohar. Isoperimetric numbers of graphs. JCT B 47, 274--291 (1989)},
\[
\D-\sqrt{\D^2-i(G)^2} \le \l_2(L) \le 2i(G).
\]
As $1-x/2 \ge \sqrt{1-x}$, 
\[
\l_2(L) \ge \D \brac{1-\sqrt{1-\frac{i(G)^2}{\D^2}}}\ge \frac{i(G)^2}{2\D}.
\]
Thus, together with the above inequality and  \cref{pro:Colin}.(ii) we obtain (i).

We now prove (ii)
As $M(t)$ is a Martingale  we have
\[
\V(M(t))\leq \sum_{s=0}^{t-1} \E \left((M(s+1)-M(s))^2
\right).
\]
Its obvious that $(Y'-Y)^2 \le \left(\frac{\Delta}{2m}  K\right)^2$. 

\[
\V(M(t))\leq t\left(\frac{\Delta}{2m}  K\right)^2.
\]

Now we are ready to proof (iii).
Let
$Y=\sum_{i=1}^n \xi_i$ and  $Y'=\sum_{i=1}^n \xi'_i$
As $Y$ is a Martingale (\cref{pro:Colin}) we have
\[
\V(Y(t))\leq \sum_{s=0}^{t-1} \E \left((Y(s+1)-Y(s))^2\right).
\]
Its obvious that $(Y'-Y)^2 \le K^2$. Thus as $\avg(t)=Y(t)/n$,
\[
\V(\avg(t))=\frac{1}{n^2} \V(Y(t))= \frac{t}{n^2} K^2.
\]

\end{proof}

\newpage
\section{The \texorpdfstring{\dualprocess}{}, example with $k=2$} \label{app:secondExample}
\begin{figure}[b]
\centering
\begin{minipage}{.41\textwidth}
    \includegraphics[width=.9\textwidth]{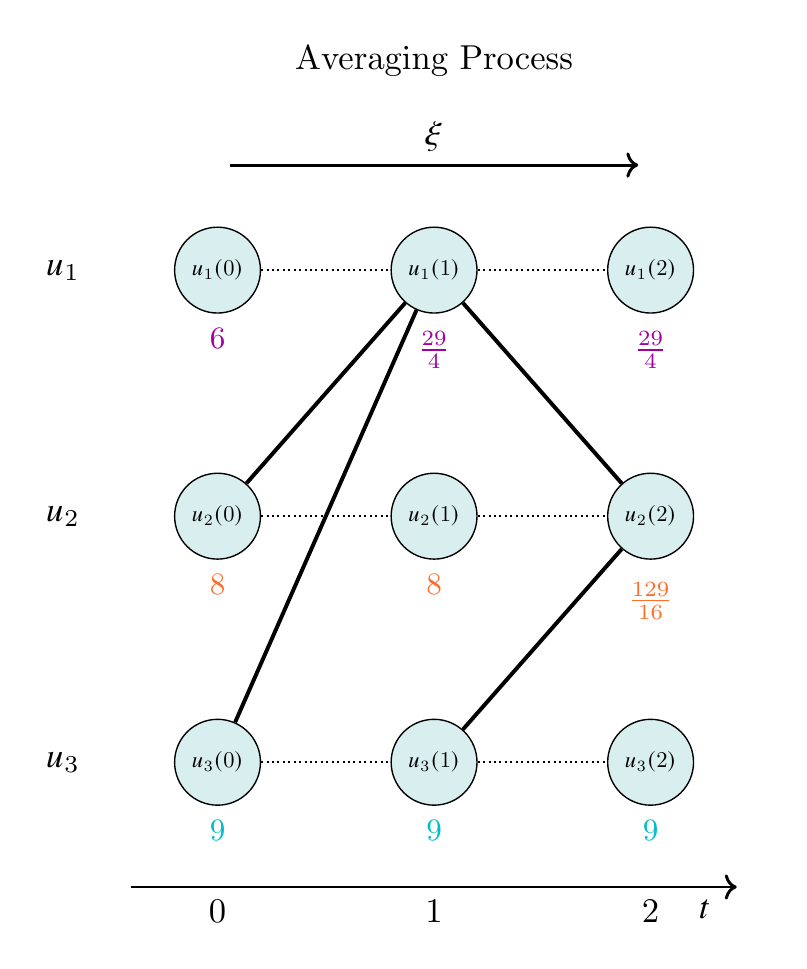}
    \caption*{(a)}
\end{minipage}
\begin{minipage}{.41\textwidth}
    \includegraphics[width=.8\textwidth]{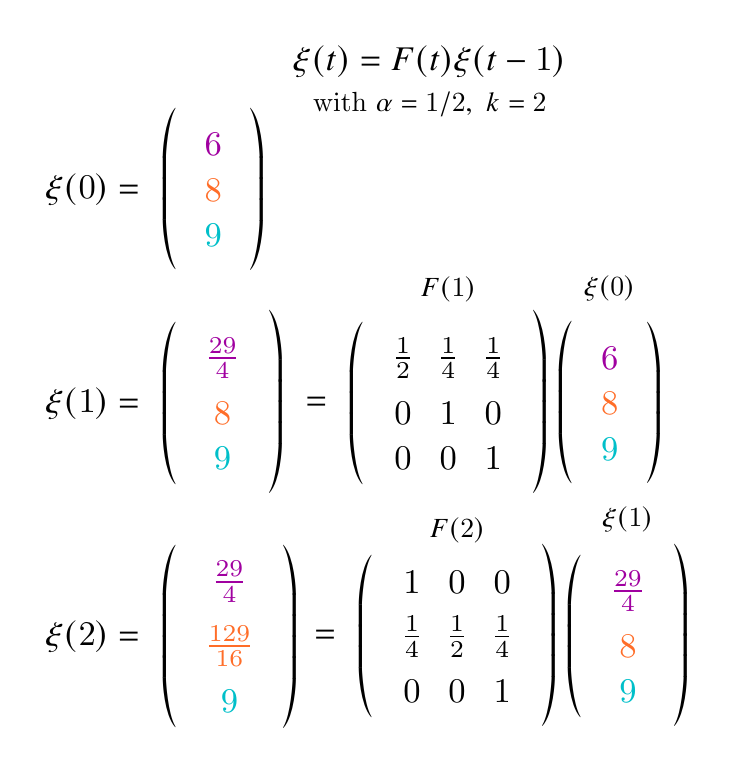}
\end{minipage}
\begin{minipage}{.41\textwidth}
    \includegraphics[width=.9\textwidth]{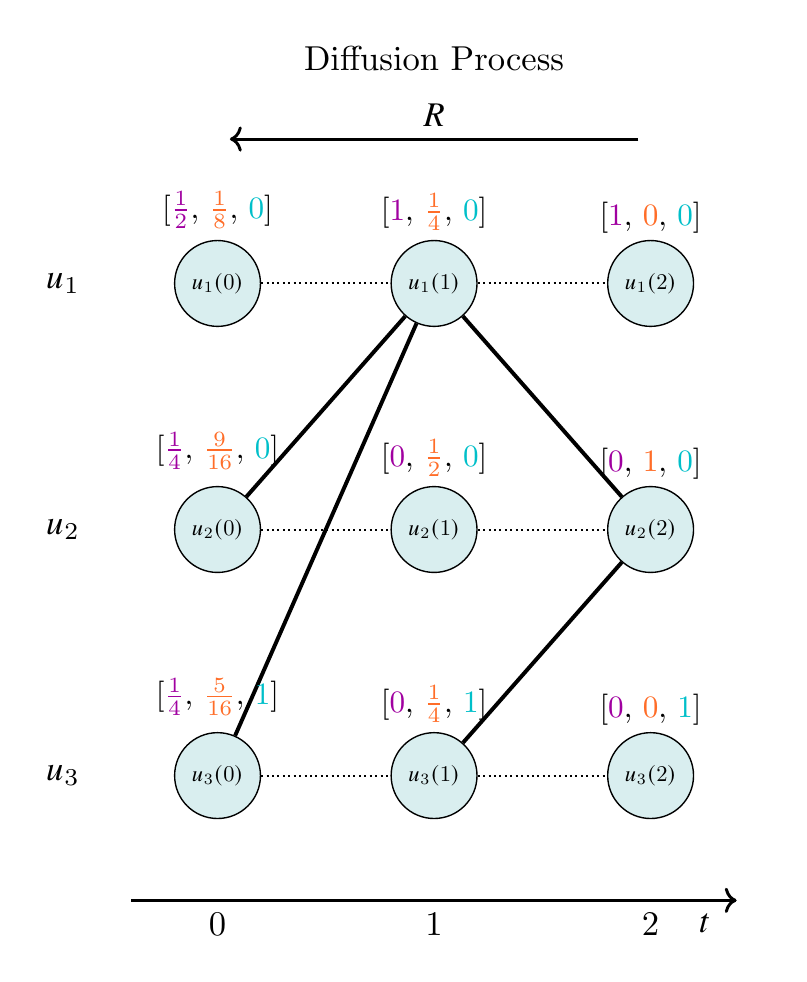}
    \caption*{(b)}
\end{minipage}
\begin{minipage}{.41\textwidth}
    \includegraphics[width=.9\textwidth]{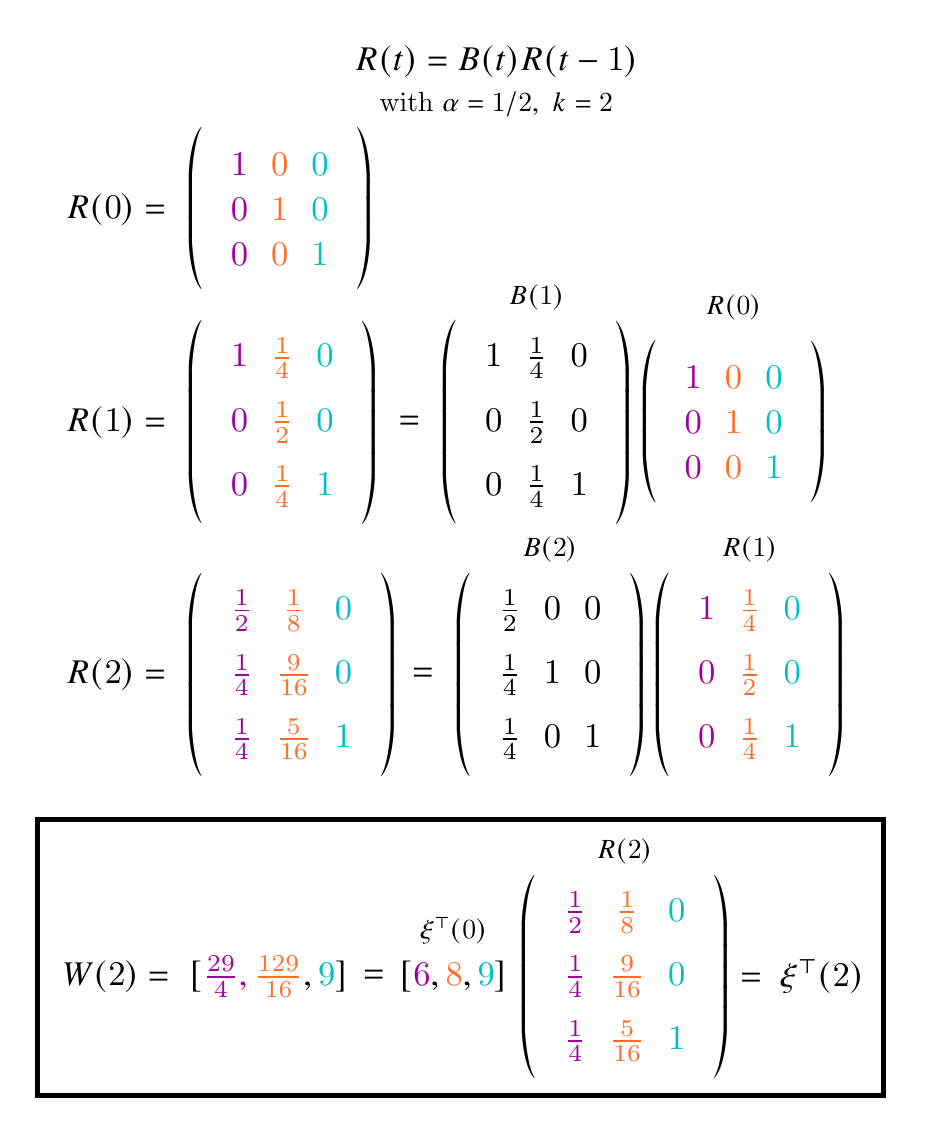}
\end{minipage}
\caption{Illustration of the duality with $k=2$ and $\alpha=1/2$.
In (a), at $t=1$ node $u_1$ is selected to update its value via averaging with $u_2$ and $u_3$. The values at $u_2$ and $u_3$ stay the same -- see matrix $F(1)$. At $t=2$, $u_2$ is selected to average with $u_1$ and $u_3$, leading to $\xi(2)$
-- see matrix $F(2)$.
In (b), the initial state, at $t=2$, of the diffusion starting in $u_2$ is the vector $[{\color{csecond}0,1,0}]$. 
At the first step ($t=2$), $u_2$ sends 
$1-\alpha = 1/2$ of its load to $u_1$ and $u_3$. 
The loads in the other nodes do not spread. The resulting load vector is $R_2(1) = [{\color{csecond}1/4, 1/2, 1/4}]$, the second column of $R(1)$.
After the second step, the load is $R_2(2) = [{\color{csecond}1/8, 9/16, 5/16}]$. The diffusion of the loads originating 
at nodes $u_1$ and $u_3$ is indicated in {\color{cfirst}\cfirst} and {\color{cthird}\cthird} respectively.
We get $W(2) = \xi^\top(2)$.}
\label{fig:couplingk2}
\end{figure}






\end{document}